\documentclass[a4paper,10pt, reqno]{amsart}

\usepackage{amsmath,amssymb,latexsym, amsfonts}
\usepackage{verbatim}
\usepackage{times}
\usepackage{mathbbol}
\usepackage{hyperref}
\usepackage{mathabx}
\usepackage{bbm}

\usepackage{tikz}

\usepackage{stmaryrd}

\usepackage{capt-of}

\newcommand{\compactlist}[1]{\setlength{\itemsep}{0pt} \setlength{\parskip}{0pt} \setlength{\leftskip}{-0.#1em}}

\newcommand{\lie}{\mathcal{L}}
\newcommand{\BB}{B}
\newcommand{\SSS}{{S}}
\newcommand{\bb}{{b}}

\newcommand{\dd}{{d}}
\newcommand{\sss}{{s}}
\newcommand{\ttt}{{t}}

\numberwithin{equation}{section}

\theoremstyle{plain}

\newtheorem{theorem}{Theorem}[section]

\newtheorem{prop}[theorem]{Proposition}

\newtheorem{lem}[theorem]{Lemma}
\newtheorem{corollary}[theorem]{Corollary}

\theoremstyle{definition}
\newtheorem{definition}[theorem]{Definition}
\newtheorem{example}[theorem]{Example}

\newtheorem{rem}[theorem]{Remark}

\newcommand{\ahha}{{\scriptscriptstyle{A}}}

\newcommand{\emme}{{\scriptscriptstyle{M}}}
\newcommand{\enne}{{\scriptscriptstyle{N}}}

\newcommand{\uhhu}{{\scriptscriptstyle{U}}}

\newcommand{\vauu}{{\scriptscriptstyle{V}}}
\newcommand{\ikks}{{\scriptscriptstyle{X}}}

\newcommand{\gb}{\beta}  

\newcommand{\gd}{\delta} 
\newcommand{\gD}{\Delta} 
\newcommand{\gve}{\varepsilon} 
  
\newcommand{\gvf}{\varphi}

\newcommand{\gO}{\Omega}

\newcommand{\gs}{\sigma} 
\newcommand{\gS}{\Sigma}
 
\newcommand{\gvt}{\vartheta} 

\newcommand{\cM}{{\mathcal M}}

\newcommand{\cO}{{\mathcal O}}

\newcommand{\Hom}{\operatorname{Hom}}

\newcommand{\Tor}{{\rm Tor}}
\newcommand{\Ext}{{\rm Ext}}

\newcommand{\id}{{\rm id}}
\newcommand{\im}{{\rm im}\,}

\newcommand{\due}[3]{{}_{{#2 }} {#1}_{{ #3}}\,}    

\newcommand{\pl}{\partial}

\newcommand{\rmref}[1]{{\rm (}\ref{#1}{\rm )}}

\newcommand{{\Hl}}{{H^{\ell}}} 
\newcommand{{\mHop}}{{m_{H^{\rm op}}}} 
\newcommand{{\Hop}}{{H^{\rm op}}} 
\newcommand{{\mUop}}{{m_{U^{\rm op}}}} 
\newcommand{{\mUopp}}{{m_{\scriptscriptstyle{U^{\rm op}}}}} 
\newcommand{{\Uop}}{{U^{\rm op}}}
\newcommand{{\mVop}}{{m_{V^{\rm op}}}} 
\newcommand{{\Vop}}{{V^{\rm op}}}  
\newcommand{{\Ae}}{{A^{\rm e}}}
\newcommand{{\Ue}}{{U^{\rm e}}}
\newcommand{{\He}}{{H^{\rm e}}}
\newcommand{{\Aop}}{{A^{\rm op}}}
\newcommand{{\Aope}}{({A^{\rm op}})^{\rm e}}
\newcommand{{\Aopl}}{{A^{\rm op}_\pl}}

\newcommand{{\Bop}}{{B^{\rm op}}}
\newcommand{{\Bope}}{({B^{\rm op}})^{\rm e}}
\newcommand{{\Bpl}}{{B_\pl}}

\newcommand{{\op}}{{{\rm op}}}
\newcommand{{\coop}}{{{\rm coop}}}
\newcommand{{\sop}}{{*^{\rm op}}}

\newcommand{\amoda}{A^{\rm e}\mbox{-}\mathbf{Mod}}                  %
\newcommand{\moda}{A^\mathrm{op}\mbox{-}\mathbf{Mod}}         %
\newcommand{\umod}{U\mbox{-}\mathbf{Mod}}                     
\newcommand{\modu}{U^\mathrm{op}\mbox{-}\mathbf{Mod}}         %
\newcommand{\yd}{{}^\uhhu_\uhhu\mathbf{YD}}                     
\newcommand{\ayd}{{}^\uhhu\mathbf{aYD}_\uhhu}                     

\newcommand{\lact}{\smalltriangleright}                  
\newcommand{\ract}{\smalltriangleleft}
\newcommand{\blact}{\blacktriangleright}  
\newcommand{\bract}{\blacktriangleleft}

\newcommand{{\gog}}{{G \rightrightarrows G_0}}

\newcommand{{\rra}}{\rightrightarrows}

\newcommand{{\lra}}{\ \longrightarrow \ }
\newcommand{{\lla}}{\ \longleftarrow \ }
\newcommand{{\lma}}{\ \longmapsto \ }

\newcommand{{\bull}}{{\scriptscriptstyle{\bullet}}}
\newcommand{{\qqquad}}{{\quad\quad\quad}}
\newcommand{\Aopp}{{\scriptscriptstyle{\Aop}}}
\newcommand{\Aee}{{\scriptscriptstyle{\Ae}}}

\begin{document}

\title{Gerstenhaber and Batalin-Vilkovisky structures on modules over operads} 

\author{Niels Kowalzig}

\begin{abstract}
In this article, we show under what additional ingredients 
a {comp} (or {opposite}) module
over an operad with multiplication
can be given the structure of a cyclic $k$-module and how the underlying simplicial homology 
gives rise to a Batalin-Vilkovisky module over the cohomology of the operad. 
In particular, one obtains a generalised Lie derivative and a generalised (cyclic) cap product that obey a Cartan-Rinehart homotopy formula, and hence yield the structure of a noncommutative differential calculus in the sense of Nest, Tamarkin, Tsygan, and others. 
Examples include the calculi known for the Hochschild theory of associative algebras, for Poisson structures, but above all the calculus for general left Hopf algebroids with respect to general coefficients (in which the classical calculus of vector fields and differential forms is contained).
\end{abstract}

\address{Istituto Nazionale di Alta Matematica, P.le Aldo Moro 5, 00185 Roma, Italia}

\email{kowalzig@mat.uniroma2.it}



\maketitle

\tableofcontents

\section{Introduction}
Higher structures,
such as those of Gerstenhaber or Batalin-Vilkovisky algebras resp.\ modules, 
 on the cohomology and homology groups of a given mathematical object---or, more generally, the corresponding homotopy structures on the underlying (co)chain complexes in connection with Deligne and formality conjectures---have been intensively investigated in recent years, in many cases in connection to operad theory or cyclic homology; see, for example, \cite{BraLaz:HBVAIPG, GalTonVal:HBVA, GerSch:ABQGAAD, GetJon:OHAAIIFDLS, GetKap:COACH, GerVor:HGAAMSO, Kau:APOACVODCVC, KauWarZun:TOOFGBVATME, KonSoi:NOAA, Kow:BVASOCAPB, KowKra:BVSOEAT, Kra:DOBVA, MarShnSta:OIATAP, McCSmi:ASODHCC, Men:BVAACCOHA, Men:CMCMIALAM, TamTsy:NCDCHBVAAFC, Tsy:FCFC, Vor:HGA, Vor:HDBAHA}, and references therein.

\subsection{Aims and objectives}

The main goal of this article is to detect the minimal structure required on the pair of an operad with multiplication $\cO$ 
(in the category of $k$-modules, where $k$ is a commutative ring) and of a (certain) 
module $\cM$ over $\cO$ to allow, when descending to the cohomology resp.\ homology associated to $\cO$ resp.\ $\cM$, for the full structure of a noncommutative differential calculus. This means (see Definition \ref{golfoaranci} for the full account), one wants a pair of a Gerstenhaber algebra $V$ and a Gerstenhaber module $\gO$ which is even Batalin-Vilkovisky, {\em i.e.}, a graded space $\gO_\bull$ that carries the structure of a graded module over $V^\bull$ (with action denoted by $\iota$), of a graded Lie algebra module over $V^{\bull+1}$ (with action denoted by $\lie$) obeying a Leibniz rule that mixes $\lie$ and $\iota$, plus a differential $B: \gO_\bull \to \gO_{\bull+1}$ that yields a (Cartan-Rinehart) homotopy formula connecting the three operators $\lie$, $\iota$, and $B$. One of the main steps in the construction is to add more structure in form of a cyclic operator $t$ and an ``extra'' comp module map $\bullet_0$ (see below for details) so as to 
obtain the structure of a simplicial and cyclic $k$-module on $\gO$;
the operator $B$ then is nothing else than the (Connes-Rinehart-Tsygan) cyclic coboundary (as the notation suggests).

In \cite{KowKra:BVSOEAT} such a theory was recently discussed on a relatively general level by making use of the notion of bialgebroids, and the question arose whether even the bialgebroid framework can be replaced by a more abstract setting. This article shows that, in fact, there is a more general structure behind.
However, whereas on the bialgebroid level the underlying simplicial structure on $\gO$ can be naturally given without asking for a cyclic operator $t$, in our construction we already need $t$ to define the ``last'' face map $d_n$ of the simplicial $k$-module structure on $\gO$. On the other hand, this is perhaps not too surprising: 
in Hochschild homology \cite{Hoch:OTCGOAAA} for a $k$-algebra $A$,
some sort of cyclic permutation in $d_n$ has tacitly taken place (which may be considered as a consequence of the fact that Hochschild homology can be computed---if $A$ is $k$-projective---by the (see \cite{Qui:CCAAE})
{\em cyclic tensor product} $M \otimes_\Aee \mathrm{Bar}(A)$ between an $A$-bimodule $M$ and the bar resolution of $A$). 
The same can be said 
in the context of a
left bialgebroid $U$ 
as in \cite[Eq.~(26)]{KowKra:CSIACT}: for defining $d_n$ a right $U$-module structure 
is needed, but right $U$-modules do barely have any useful properties unless there is some sort of Hopf structure on $U$ (which is, loosely speaking, equivalent to having a cyclic operator), see again {\em op.~cit.}, Lemmata 3.2 \& 3.3, for more information, and also Eq.~\rmref{corelli1} below.

\subsection{Cyclic modules over operads with multiplication}

If $(\cO, \mu)$ is an operad with multiplication in the category of $k$-modules (where $k$ is a commutative ring, see the main text for detailed definitions and conventions we use here) and $\cM = \{\cM(n)\}_{n \geq 0}$ a unital left {\em comp} (or {\em opposite}) module 
over this operad (see Definition \ref{molck}), 
we define 
an additional
structure on $\cM$ 
by 
adding two more ingredients which allows for the structure of a cyclic $k$-module on $\cM$: an {\em extra comp module map}
$$
\bullet_0: \cO(p) \otimes \cM(n) \to \cM({n-p+1}), \quad 0 \leq p \leq n+1,
$$
that is required to fulfil the same relations as the ``standard'' comp module maps $\bullet_i, \ i \geq 1$ (see Definition \ref{molck}), plus a morphism $t: \cM(n) \to \cM(n)$ that fulfils
\begin{equation*}
\label{lagrandebellezza1a}
t(\gvf \bullet_{i} x) = \gvf \bullet_{i+1} t(x), \qquad i = 0, \ldots, n-p,  \qquad \gvf \in \cO(p), \ x \in \cM(n). 
\end{equation*}
Of particular interest is the situation in which one has
$t^{n+1} = \id$:
in such a situation we call $\cM$ a {\em cyclic (unital) comp module over $\cO$}. With these assumptions, we prove in \S\ref{ilpattodifamiglia} 
(see Proposition \ref{sabaudia} for details):

\begin{prop}
\label{pomezia}
A cyclic unital comp module over $\cO$ gives rise to a cyclic $k$-module. 
\end{prop}

This allows, in a standard fashion, to define 
not only simplicial (Hochschild) homology on $\cM$ (with underlying simplicial $k$-module structure denoted by $\cM_\bull$) 
with respect to a boundary 
operator $b: \cM_\bull \to \cM_{\bull-1}$, but also a mixed complex $(\cM, b, B)$ with coboundary operator $B: \cM_\bull \to \cM_{\bull+1}$, and hence for the definition of the cyclic homology groups of $\cM$ (see Definition \ref{avviso}). 

\subsection{The Gerstenhaber module}

In particular, once the notion of a cyclic unital comp module over an operad with multiplication is established, we proceed in \S\ref{federtasche} by introducing, for $\gvf \in \cO(p)$ and $ x \in \cM(n)$, 
a general notion of 
{\em cap product} 
$$
\iota_\gvf x := \gvf \smallfrown x := (\mu \circ_2 \gvf) \bullet_0 x,  
$$ 
and of 
{\em Lie derivative}  
\begin{equation*}
        \lie_\varphi x :=         
\sum^{n-p+1}_{i=1} 
(-1)^{(p-1)(i-1)} 
\varphi \bullet_i x 
+ 
\sum^{p}_{i=1} 
(-1)^{n(i-1) + p -1} 
\varphi \bullet_0 \ttt^{i-1} (x),  
\end{equation*}
%
%
see Definition \ref{volterra} for all details. 
Recall then from \cite{Ger:TCSOAAR, GerSch:ABQGAAD, McCSmi:ASODHCC} the well-known fact that an operad with multiplication $\cO$ gives rise to a cosimplicial space (usually denoted by $\cO^\bull$) with coboundary $\gd$ 
and cohomology $H^\bull(\cO) := H(\cO^\bull, \gd)$. Together with the {\em cup product} $\smallsmile$, defined in terms of the operad multiplication and the {\em comp maps} (Gerstenhaber products) of the Gerstenhaber algebra, one obtains a differential graded associative algebra, and with respect to this structure we prove in Proposition \ref{estintore}:

\begin{prop}
$(\cM_{-\bull},\bb,\smallfrown)$ forms a left dg module over 
$(\cO^\bull,\delta,\smallsmile)$.
\end{prop}


What is more, any operad defines a graded Lie bracket (see Eq.~\rmref{zugangskarte}), which we call (already on the cochain level) the {\em Gerstenhaber bracket} $\{.,.\}$. 
When shifting the degrees of (the cosimplicial module associated to) the operad by one, we prove in Theorem \ref{feinefuellhaltertinte} with respect to this bracket:

\begin{theorem}
The Lie derivative $\lie$ defines a dg 
Lie algebra representation of $(\cO^{\bull+1}, \{.,.\})$ on $\cM_{-\bull}$. 
\end{theorem}

By the preceding results one obtains in particular the fact that both Lie derivative and cap product descend to well-defined operators on the simplicial homology $H_\bull(\cM) := H(\cM_\bull,b)$, as soon as one acts with cocycles of the cosimplicial space $(\cO^\bull, \gd)$, and in such a case one obtains a mixed Leibniz rule between $\lie$ and $\iota$. In 
Theorem \ref{waterbasedvarnish}, we prove:

\begin{theorem}
For any two cocycles 
$\varphi$ and $\psi$, one has
\begin{equation*}
\label{radicale1a}
[\iota_\psi, \lie_\varphi] = \iota_{\{\psi, \gvf\}}
\end{equation*}
for the induced maps on homology.
\end{theorem}

These three results together mentioned in this subsection constitute the fact that $\cM$ forms a {\em Gerstenhaber module} over $\cO$, in the sense of Definition \ref{golfoaranci} below.

\subsection{The Batalin-Vilkovisky module}

For the structure of a {\em Batalin-Vilkovisky module} (as in Definition \ref{golfoaranci} below), hence for the full structure of a {\em noncommutative differential calculus} in the terminology of \cite{Tsy:CH}, one needs to take the full cyclic bicomplex into account and find what we call a {\em cyclic correction} to the cap product, given by an operator $S_\gvf$ introduced in \S\ref{gallimard}. One defines
\begin{equation*}
        S_\varphi := 
        \sum^{n-p+1}_{j=1} \, 
        \sum^{n - p +1}_{i=j} 
              (-1)^{n(j-1) + (p-1)(i-1)}               
             e \bullet_0 \big(\gvf \bullet_i t^{j-1}(x)\big),
\end{equation*}
where $e$ is part of the multiplicative structure of $\cO$ called the {\em unit}, see Definition \ref{moleskine}. With the help of this operator and the cyclic coboundary $B$ mentioned after Proposition \ref{pomezia}, we finally prove in \S\ref{lumograph} a Cartan-Rinehart homotopy formula for cyclic unital modules over operads with multiplication:

\begin{theorem}
For $\varphi$ in the normalised cochain complex $\bar{\cO}^\bull$, the homotopy formula
\begin{equation*}
\lie_\varphi = [\BB+\bb, \iota_\varphi + \SSS_\gvf] - \iota_{\gd\varphi} - \SSS_{\gd\varphi},
\end{equation*}
holds on the normalised chain complex $\bar \cM_\bull$.
In particular, for the induced maps on homology, one obtains for any cocycle $\gvf \in \bar{\cO}^\bull$ the homotopy
\begin{equation*}
\lie_\gvf = [B, \iota_\gvf],
\end{equation*}
that is, the pair $(H^\bull(\cO), H_\bull(\cM))$ defines a noncommutative differential calculus.
\end{theorem}

In more restricted situations, 
the operators $\bb, \BB, \iota, \lie$, and $\SSS$ along with relations analogous to those mentioned above appeared in the literature a couple of times before, to our knowledge starting with Rinehart \cite{Rin:DFOGCA}, where
these operators on Hochschild (co)chains for commutative associative algebras were called (in the same order) $\Delta$, $\bar d$,
$c$, $\theta$, and $f$, respectively. In noncommutative Hochschild theory, the Lie derivative for $1$-cocycles, denoted by $\delta^*$, appeared in Connes \cite[p.~124]{Con:NCDG}, and also in Goodwillie \cite{Goo:CHDATFL}, who in addition dealt with $\iota$ and $\SSS$, denoted by $e$ and $E$, respectively. Both Getzler \cite{Get:CHFATGMCICH}, where they are denoted by $\mathbf{b}$ and $\mathbf{B}$, as well as Gel'fand, Daletski{\u\i}, Nest, and Tsygan \cite{GelDalTsy:OAVONCDG, NesTsy:OTCROAA, Tsy:CH}, generalised these operators to arbitrary cochains (and also to differential graded algebras and multiple cochain entries). Finally, in \cite{KowKra:BVSOEAT} the theory was enlarged to the realm of (left) Hopf algebroids, containing both the Hochschild theory as well as the classical Cartan calculus in differential geometry as examples.

\subsection{Applications and examples}

In \S\ref{castropretorio1}, we discuss how the well-known calculus structure on Hochschild chains and cochains from the literature cited above fits into our construction, extending the theory by introducing coefficients. 
In \S\ref{castropretorio2}, we resume the discussion started in \cite{Kow:BVASOCAPB} of how the homology and cohomology of noncommutative Poisson algebras \cite{Xu:NCPA} and the Poisson bicomplex \cite{Bry:ADCFPM}  can be described by our theory, again seizing the occasion of introducing coefficients in Poisson homology. 
In \S\ref{castropretorio3} we explain how the noncommutative differential calculus on bialgebroids resp.\ left Hopf algebroids of \cite{KowKra:BVSOEAT} can be understood by the approach developed in this article, and we extend the theory in \cite{KowKra:BVSOEAT}, using a result in \cite{Kow:BVASOCAPB}, by introducing more general coefficients in the cochain space that assumes the r\^ole of the operad. The outcome is given by Theorem \ref{zanjrevolution} and Corollary \ref{dieda}:

\begin{theorem}
Let $(U,A)$ be a left Hopf algebroid, $M$ an anti Yetter-Drinfel'd module and $N$ a braided commutative Yetter-Drinfel'd algebra. Then $C_\bull(U,M \otimes_\Aopp N)$ forms a para-cyclic module over the operad $C^\bull(U,N)$ with multiplication, which is cyclic if $M \otimes_\Aopp N$ is stable.
In this case, the pair $\big(H^\bull(U,N), H_\bull(U,M \otimes_\Aopp N)\big)$ carries a canonical structure of a noncommutative differential calculus, and if $U$ is right $A$-projective, then this is true for the pair 
$\big(\Ext^\bull_\uhhu(A,N), \Tor^\uhhu_\bull(M,N)\big)$.
\end{theorem}

Again we refer to the main text in \S\ref{castropretorio3} for the necessary (and numerous) details in notation and terminology to understand this statement.

\bigskip

\thanks{ {\bf Acknowledgements.}   \!
I would like to thank Paolo Salvatore for helpful comments and discussions on operads and the referee for careful reading and suggestions. 
This research was funded by an INdAM-COFUND Marie Curie grant.

\section{Preliminaries}

Throughout the text, let $k$ be a commutative ring. As usual, an unadorned tensor product is meant to be over $k$.

\subsection{Gerstenhaber algebras and Batalin-Vilkovisky modules}

In this section, we explain the main objects of interest, that of Gerstenhaber algebras and Gerstenhaber modules, which, if equipped with some extra structure, lead to the notion of Batalin-Vilkovisky modules:

\begin{definition}
\label{golfoaranci}
\indent
\begin{enumerate}
\compactlist{99}
\item
A {\em Gerstenhaber algebra} (or {\em graded Poisson algebra})
over a commutative ring $k$ consists of a graded commutative $k$-algebra
$(V,\smallsmile)$,  
$$
        V=\bigoplus_{p \in \mathbb{N}} V^p,\quad
        \gvf \smallsmile \psi=(-1)^{pq}\psi \smallsmile \gvf
        \in V^{p+q},\quad 
        \gvf \in V^p,\psi \in V^q,
$$ 
equipped with a graded Lie bracket 
$
        \{\cdot,\cdot\} : V^{p+1} \otimes V^{q+1} \rightarrow V^{p+q+1}
$ 
on the \emph{desuspension} 
$$
        V[1]:=\bigoplus_{p \in \mathbb{N}} V^{p+1}
$$
of $V$, for which all operators of the form $\{\chi,\cdot\}$ satisfy 
$$
        \{\chi,\gvf \smallsmile \psi\}=
        \{\chi,\gvf\} \smallsmile \psi + (-1)^{pq} \gvf \smallsmile
        \{\chi,\psi\},\quad
        \chi \in V^{p+1},\gvf \in V^q,
$$
that is, the {\em graded Leibniz rule}.
\item 
A {\em Gerstenhaber module} over  
$V$ is a graded $(V,\smallsmile)$-module 
$(\Omega,\smallfrown)$,
$$
        \Omega =\bigoplus_{n \in \mathbb{N}} \Omega_n,\quad
        \gvf \smallfrown x \in \Omega_{n-p},\quad
        \gvf \in V^p,x \in \Omega_n,
$$
with a graded Lie algebra module structure over
$(V[1],\{\cdot,\cdot\})$ 
$$
        \lie : V^{p+1} \otimes \Omega_n \rightarrow \Omega_{n-p},\quad
        \gvf \otimes x \mapsto \lie_\gvf (x), 
$$
which satisfies for 
$\gvf \in V^{p+1},\psi \in V^q, x \in \Omega$ 
$$
 \psi \smallfrown \lie_\gvf(x) =  \{\psi,\gvf\} \smallfrown  x +(-1)^{pq} \lie_\gvf(\psi \smallfrown x),
$$ 
that is, the {\em mixed Leibniz rule}.
\item 
Such a module  
is called {\em Batalin-Vilkovisky} module if one has a $k$-linear differential
$$
        \BB : \Omega_n \rightarrow \Omega_{n+1},\quad \BB  \BB=0,
$$
such that $\lie_\gvf$ for $\gvf \in V^p$ can be expressed as 
$$
        \lie_\gvf(x) =\BB(\gvf \smallfrown x)-(-1)^p \gvf \smallfrown 
        \BB(x), 
$$
that is, the {\em (Cartan-Rinehart) homotopy formula}.
\item 
A pair $(V,\Omega)$ 
of a Gerstenhaber algebra and of a Batalin-Vilkovisky module 
over it is called a {\em (noncommutative) differential calculus}.
\end{enumerate}
\end{definition}

Observe that the $(V, \smallsmile)$-module $(\gO, \smallfrown)$ in the second paragraph is actually a graded module in {\em negative} degrees as $\gvf \smallfrown -$ lowers degrees by $p$; since this is the only situation we shall 
deal with we do not consider this sloppiness in terminology as a major problem. The first two paragraphs in the above definition define what was called a {\em precalculus} in \cite[Def.~4.3]{Tsy:CH}, and adding the structure in (iii) to a precalculus then makes a {\em calculus}. We would also like to point to the short discussion after Definition 1.1 in \cite{KowKra:BVSOEAT} on various conventions regarding variants of the above definition and terminology.


\pagebreak

\subsection{Operads and comp modules} 
\label{responsabilitacivile}
\subsubsection{Operads with multiplication}
In this subsection, we gather some well-known material about operads; for deeper information see, for example, \cite{Fre:MOOAF, GerSch:ABQGAAD, LodVal:AO, Mar:MFO, Mar:SAAC, MarShnSta:OIATAP, Men:BVAACCOHA}, 
and references therein. 
Here, we essentially mention only the notion of nonsymmetric operads with multiplication and certain modules over them, as well as their connection to Gerstenhaber algebras in Theorem \ref{customerscopy}:

\begin{definition}
\label{moleskine}
A {\em non-$\gS$ operad} $\cO$ in the category 
of $k$-modules is a sequence $\{\cO(n)\}_{n \geq 0}$ of $k$-modules 
with $k$-bilinear operations $\circ_i: \cO(p) \otimes \cO(q) \to \cO({p+q-1})$, sometimes called {\em comp maps} or {\em Gerstenhaber products} or still {\em partial composites}, along with an identity $\mathbb{1} \in \cO(1)$, such that
\begin{eqnarray}
\label{danton}
\nonumber
\gvf \circ_i \psi &=& 0 \qquad \qquad \qquad \qquad \qquad \! \mbox{if} \ p < i \quad \mbox{or} \quad p = 0, \\
(\varphi \circ_i \psi) \circ_j \chi &=& 
\begin{cases}
(\varphi \circ_j \chi) \circ_{i+r-1} \psi \qquad \mbox{if} \  \, j < i, \\
\varphi \circ_i (\psi \circ_{j-i +1} \chi) \qquad \hspace*{1pt} \mbox{if} \ \, i \leq j < q + i, \\
(\varphi \circ_{j-q+1} \chi) \circ_{i} \psi \qquad \mbox{if} \ \, j \geq q + i,
\end{cases} \\
\nonumber
\gvf \circ_i \mathbb{1} &=& \mathbb{1} \circ_1 \gvf \ \ = \ \ \gvf \! \qquad \quad \qquad \mbox{for} \ i \leq p   
\end{eqnarray}
holds true for any $\varphi \in \cO(p), \, \psi \in \cO(q)$, and $\chi \in \cO(r)$. 
The operad is called {\em operad with multiplication} if there exists a {\em distinguished element} or {\em operad multiplication} $\mu \in \cO(2)$ along with an element $e \in \cO(0)$, the {\em unit}, such that
\begin{eqnarray}
\label{distinguished1}
\mu \circ_1 \mu &=& \mu \circ_2 \mu, \\
\label{distinguished2}
\mu \circ_1 e &=& 
\mu \circ_2 e
\ \ = \ \ \mathbb{1} 
\end{eqnarray}
is fulfilled.
\end{definition}
In what follows, {\em operad} will always mean non-$\gS$ operad in the category of $k$-modules as in the above definition.

An operad with multiplication connects to the notion of Gerstenhaber algebra 
by the following well-known result \cite{Ger:TCSOAAR, GerSch:ABQGAAD, McCSmi:ASODHCC}:

\begin{theorem}
\label{customerscopy}
Each operad with multiplication defines a cosimplicial $k$-module. Its cohomology forms a Gerstenhaber algebra.
\end{theorem}

The structure maps in the explicit construction of this Gerstenhaber structure read as follows:
for any two cochains $\varphi \in \cO(p),\psi \in \cO(q)$, set
\begin{equation*}
        \varphi \bar\circ \psi := 
\sum^{p}_{i=1}
        (-1)^{(q-1)(i-1)} \varphi \circ_i \psi \in \cO({p+q-1}), 
\end{equation*}
and define their \emph{Gerstenhaber bracket} by
\begin{equation}
\label{zugangskarte}
{\{} \varphi,\psi \}
:= \varphi \bar\circ \psi - (-1)^{(p-1)(q-1)} \psi \bar\circ \varphi.
\end{equation}
The pertinent graded commutative product is given by the {\em cup product}
\begin{equation}
\label{cupco}
        \varphi \smallsmile \psi := 
        (\mu \circ_2 \gvf) \circ_1 \psi \in \cO({p+q}).
\end{equation}
Observe that we deploy here a convention which is opposite to the cup product originally introduced in \cite{Ger:TCSOAAR}.

\noindent Finally, the coboundary of the cosimplicial $k$-module (denoted $\cO^\bull$, where $\cO^n := \cO(n)$) 
results as 
\begin{equation}
\label{erfurt}
        \delta \varphi = \{\mu,  \varphi\},
\end{equation}
and then the triple $(\cO^\bull, \gd, \smallsmile)$ forms a dga (differential graded associative) algebra.

Theorem \ref{customerscopy} will be a crucial ingredient in the next section. A stronger variant is a similar 
relation between Batalin-Vilkovisky {\em algebras} and {\em cyclic} operads \cite{GetKap:COACH} (with multiplication) established 
in \cite{Men:BVAACCOHA}, which is, however, not within the direct scope of this article.

\subsubsection{Comp modules over operads} 
Our guiding example in later considerations is to consider tensor chains with the endomorphism operad acting on it in a certain way (see Example \ref{schnee1} below), which essentially leads to what we called {\em comp module} in \cite[\S4.2]{KowKra:BVSOEAT} and which is the structure we need throughout:



\begin{definition}
\label{molck}
A {\em (left) comp module $\cM$ 
over an operad $\cO$} (or an {\em opposite $\cO$-module}) is
a sequence of $k$-modules $\cM = \{ \cM(n) \}_{n \geq 0}$ together with $k$-linear
operations, sometimes called {\em comp module maps}, 
$$
        \bullet_i : 
        \cO(p) \otimes \cM(n) \to \cM({n-p+1}), \quad \mbox{for} \ 
i = 1, \ldots, n- p +1, \quad 0 \leq p \leq n, 
$$
declared to be zero if $p > n$, and 
satisfying
for $\gvf \in \cO(p)$, $\psi \in \cO(q)$, and $x \in
\cM(n)$ the identities  
\begin{equation}
\label{TchlesischeStr}
\gvf \bullet_i \big(\psi \bullet_j x\big) = 
\begin{cases} 
\psi \bullet_j \big(\gvf \bullet_{i + q - 1}  x\big) \quad & \mbox{if} \ j < i, 
\\
(\gvf \circ_{j-i+1} \psi) \bullet_{i}  x \quad & \mbox{if} \ j - p < i \leq j, \\
\psi \bullet_{j-p + 1} \big(\gvf \bullet_{i}  x\big) \quad & \mbox{if} \ 1 \leq i \leq j - p,
\end{cases}
\end{equation}
where $p > 0$, $q \geq 0$, $ n \geq 0$. 
In case $p=0$, the index $i$ runs from $1$ to $n+1$, and the above relations need to be read as
 \begin{equation}
\label{TchlesischeStr0}
\gvf \bullet_i \big(\psi \bullet_j x\big) = 
\begin{cases} 
\psi \bullet_j \big(\gvf \bullet_{i + q - 1}  x\big) \quad & \mbox{if} \ j < i, 
\\
\psi \bullet_{j+1} \big(\gvf \bullet_{i}  x\big) \quad & \mbox{if} \ 1 \leq i \leq j.
\end{cases}
\end{equation}
A comp module over $\cO$ is called {\em unital} if  
 \begin{equation}
\label{TchlesischeStr-1}
\mathbb{1} \bullet_i x = x, \quad \mbox{for} \ i = 1, \ldots, n,
\end{equation}
for all $x \in \cM(n)$.
\end{definition}

\begin{example}
\label{schnee1}
Let $X$ be a $k$-module and $\mathcal{E}\hspace*{-1pt}{nd}_\ikks$ be the endomorphism operad defined by  $\mathcal{E}\hspace*{-1pt}{nd}_\ikks(p) := \Hom(X^{\otimes p}, X)$ with identity element $\mathbb{1} := \id_\ikks$. A unital comp module over this operad is defined by $\cM_\ikks(n) := X^{\otimes n +1 }$ along with
comp module maps given, for  $i= 1, \ldots, n -p +1$, by
$$
\gvf \bullet_i (x_0, \ldots, x_n) := (x_0, \ldots, x_{i-1}, \gvf(x_{i}, \ldots, x_{i+p-1}), x_{i+p}, \ldots, x_n),
$$
where $\gvf \in \mathcal{E}\hspace*{-1pt}{nd}_\ikks(p)$ and $x := (x_0, \ldots, x_n) := x_0 \otimes \cdots \otimes x_n \in \cM_\ikks(n)$. The case where $X$ is a $k$-algebra, {\em i.e.}, when  $\mathcal{E}\hspace*{-1pt}{nd}_\ikks$ becomes an operad with multiplication, is discussed at length in \S\ref{castropretorio1}.
\end{example}

\begin{rem}
As is clear from the preceding example, a picture that one could give 
for illustrating comp modules looks like:

\begin{center}
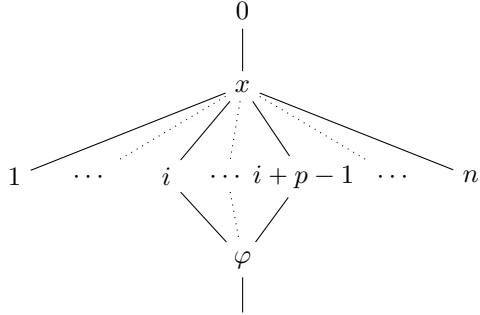

\begin{tikzpicture}
\node(0) at (0, 0) {$0$};
\node(x) at (0,-1) {$x$};
\node(1) at (-3,-2.2) {$1$};
\node(2) at (-2,-2.2) {$\cdots$};
\node(i) at (-1,-2.2) {$i$};
\node(4) at (-0.2,-2.2) {$\cdots$};
\node(i+p-1) at (0.8,-2.2) {$i+p-1$};
\node(6) at (2,-2.2) {$\cdots$};
\node(n) at (3,-2.2) {$n$};
\node(phi) at (0,-3.3) {$\varphi$};

\draw(x)--(1);
\draw[dotted](x)--(2);
\draw(x)--(i);
\draw[dotted](x)--(4);
\draw(x)--(i+p-1);
\draw[dotted](x)--(6);
\draw(x)--(n);
\draw(i)--(phi);
\draw[dotted](4)--(phi);
\draw(i+p-1)--(phi);
\draw(0)--(x);

\draw(phi)--(0,-4);
\end{tikzpicture}
\captionof{figure}{Comp modules}\label{tratto}
\end{center}
Motivated by this picture, one might want to call a comp module over an operad $\cO$ also an {\em opposite $\cO$-module}.
\end{rem}

\begin{rem}
We call such an object a {\em left} comp module since if one defines 
$\lambda_\gvf x := \gvf \bar \bullet x := \sum^{n-p+1}_{i=1} (-1)^{(p-1)(i-1)} \gvf \bullet_i x$, then one easily checks that $\lambda_\gvf \lambda_\psi x - (-1)^{(p-1)(q-1)} \lambda_\psi \lambda_\gvf x = \lambda_{\{\gvf, \psi\}} x$. By a similar argumentation, an operad as determined by the identities in \rmref{danton} can be shown to be a {\em right} module over itself (in the standard sense of a module over an operad \cite{Mar:MFO, Mar:SAAC}), which is why they were called {\em right comp algebras} in \cite[p.~63]{GerSch:ABQGAAD}.
\end{rem}

\begin{rem}
\label{infoutili}
Observe that, in contrast to the standard definitions for modules over operads (see \cite{Mar:MFO, Mar:SAAC}), 
we explicitly allow operations for the extreme cases $p = 0$ (acting on elements in $\cM$ for any $n$) or $n = 0$ (in which case, by definition, only elements in $\cO(0)$, {\em i.e.}, for $p=0$ have a non-vanishing outcome). Of course, the middle line in \rmref{TchlesischeStr} can also be read
from right to left so as to understand how an element $\gvf \circ_j
\psi$ acts on $\cM$ via $\bullet_i$. Also, giving both the first and the third line in \rmref{TchlesischeStr} is, as in \rmref{danton}, actually redundant but convenient to have at hand in explicit computations.
\end{rem}

\section{Cyclic comp modules over operads and their homology}
\label{ilpattodifamiglia}
With the preparations of the preceding section, we can now introduce the principal notion that implies the subsequent structures and results. 

\begin{definition}
\label{piantadiroma}
A {\em para-cyclic (unital, left) comp module} 
over an operad $\cO$ is a (unital, left) comp module $\cM$ over $\cO$ as in Definition \ref{molck} equipped with two additional structures:
first, an {\em extra} ($k$-linear) comp module map
$$
\bullet_0: \cO(p) \otimes \cM(n) \to \cM({n-p+1}), \quad 0 \leq p \leq n+1,
$$
declared to be zero if $p > n+1$, 
and such that the relations
\rmref{TchlesischeStr}--\rmref{TchlesischeStr-1} are fulfilled for $i=0$ as well, 
{\em i.e.}, 
for $\gvf \in \cO(p)$, $\psi \in \cO(q)$, and $x \in \cM(n)$, with $q \geq 0$, $n \geq 0$, 
\begin{eqnarray}
\label{SchlesischeStr}
\gvf \bullet_i \big(\psi \bullet_j x\big) \!\!\!\!&=\!\!\!\!& 
\begin{cases} 
\psi \bullet_j \big(\gvf \bullet_{i + q - 1}  x\big)  & \mbox{if} \ j < i, 
\\
(\gvf \circ_{j-i+1} \psi) \bullet_{i}  x  & \mbox{if} \ j - p < i \leq j, \\
\psi \bullet_{j-p + 1} \big(\gvf \bullet_{i}  x\big)  & \mbox{if} \ 0 \leq i \leq j - p,
\end{cases} \qquad \mbox{(cases for $p > 0$)\qqquad}
\\
\label{SchlesischeStr0}
\gvf \bullet_i \big(\psi \bullet_j x\big) \!\!\!\!&=\!\!\!\!& 
\begin{cases} 
\psi \bullet_j \big(\gvf \bullet_{i + q - 1}  x\big)  & \mbox{if} \ j < i, 
\\
\psi \bullet_{j+1} \big(\gvf \bullet_{i}  x\big)  & \mbox{if} \ 0 \leq i \leq j,
\end{cases}
\qqquad\quad  \mbox{(cases for $p = 0$)}
\\
\label{SchlesischeStr-1}
\mathbb{1} \bullet_i x \!\!\!\!&=\!\!\!\!&  x \hspace*{3cm} \, \mbox{for \ } i = 0, \ldots, n,
\end{eqnarray}
and second, a morphism $t: \cM(n) \to \cM(n)$ for all $ n \geq 1$ such that
\begin{equation}
\label{lagrandebellezza1}
t(\gvf \bullet_{i} x) = \gvf \bullet_{i+1} t(x),  \qquad i = 0, \ldots, n-p,
\end{equation}
holds for $\gvf \in \cO(p)$ and $x \in \cM(n)$. 
A para-cyclic comp module over $\cO$ is called {\em cyclic} if
\begin{equation}
\label{lagrandebellezza2}
t^{n+1} = \id
\end{equation}
is true on $\cM(n)$.
\end{definition}

\begin{example}
\label{schnee2}
In the situation of Example \ref{schnee1}, the structure of a cyclic unital comp module on $\cM_\ikks$ over  $\mathcal{E}\hspace*{-1pt}{nd}_\ikks$ is simply given by
$
\gvf \bullet_0 (x_0, \ldots, x_n) := \big(\gvf(x_0, \ldots, x_{p-1}), x_p, \ldots, x_n\big),
$
along with 
$
t(x_0, \ldots, x_n) := (x_n, x_0, \ldots, x_{n-1}).
$

More (and more complicated) examples will be seen in \S\ref{examples}, and this particular one will be enhanced in \S\ref{castropretorio1} in case $X$ is a $k$-algebra.
\end{example}

\begin{rem}
Examining Figure \ref{tratto} in the context of Example \ref{schnee2} again, one could interpret $t$ analogously to the ``rotating'' operator of cyclic operads, which transforms the input of $x$ into the first output and the last output into the input.
The condition \rmref{lagrandebellezza1} can then be understood by a picture similar to those describing cyclic operads:

\begin{center}
\begin{tikzpicture}
\node(0) at (-6, -3) {$0$};
\node(x) at (-4,-1.4) {$x$};
\node(2) at (-5.5,-3) {$\cdots$};
\node(i) at (-5,-3) {$i$};
\node(4) at (-4.6,-3) {$\cdots$};
\node(i+p-1) at (-3.7,-3) {$i+p-1$};
\node(6) at (-2.7,-3) {$\cdots$};
\node(n-1) at (-2,-3) {$n-1$};
\node(n) at (-4,0) {$n$};
\node(phi) at (-4.5,-4.3) {$\varphi$};

\draw[dotted](x)--(2);
\draw(x)--(i);
\draw[dotted](x)--(4);
\draw(x)--(i+p-1);
\draw[dotted](x)--(6);
\draw(x)--(n-1);
\draw(i)--(phi);
\draw[dotted](4)--(phi);
\draw(i+p-1)--(phi);

\draw(phi)--(-4.5,-5.3);




\draw  (-3.7,-1.5) to [out=330,in=270] (-1.5,-2) to [out=90,in=270] (-4,-0.2) ;
\draw (x) to [out=90, in=38] (-5.8, -2.8);

\node(0) at (1, -3) {$0$};
\node(x) at (3,-1.4) {$x$};
\node(2) at (1.5,-3) {$\cdots$};
\node(i+1) at (2.2,-3) {$i+1$};
\node(4) at (2.85,-3) {$\cdots$};
\node(i+p) at (3.5,-3) {$i+p$};
\node(6) at (4.3,-3) {$\cdots$};
\node(n-1) at (5,-3) {$n-1$};
\node(n) at (3,0) {$n$};
\node(phi) at (2.8,-4.3) {$\varphi$};
\node(=) at (-0.3, -3) {$=$};

\draw[dotted](x)--(2);
\draw(x)--(i+1);
\draw[dotted](x)--(4);
\draw(x)--(i+p);
\draw[dotted](x)--(6);
\draw(x)--(n-1);
\draw(i+1)--(phi);
\draw[dotted](4)--(phi);
\draw(i+p)--(phi);

 \draw(phi)--(2.8,-5.3);



\draw  (3.3,-1.5) to [out=330,in=270] (5.5,-2) to [out=90,in=270] (3,-0.2) ;
\draw (x) to [out=90, in=38] (1.2, -2.8);

\draw [thick, rounded corners] (-6.3, -0.5) rectangle (-1.2, -4.8);
\draw [thick, rounded corners] (0.7, -0.5) rectangle (5.8, -3.7);

 \end{tikzpicture}
 
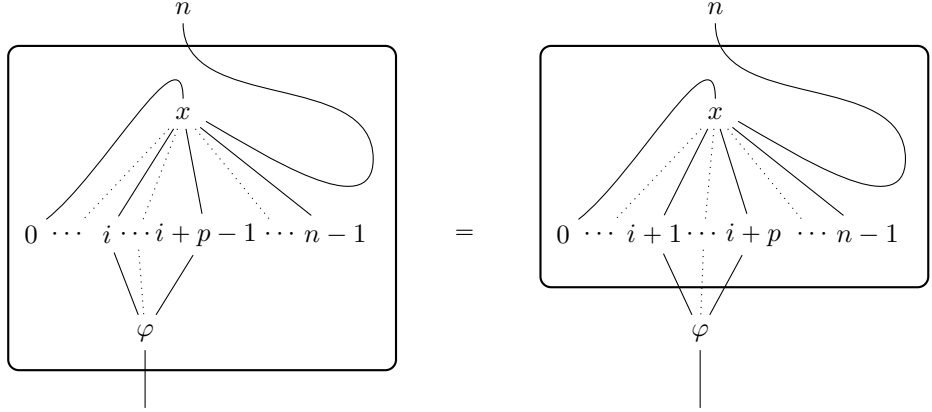
\captionof{figure}{The relation $t(\gvf \bullet_{i} x) = \gvf \bullet_{i+1} t(x)$ for cyclic comp modules}\label{rel}
 \end{center}
\end{rem}

\begin{rem}
We want to underline that, in contrast to the comp module maps $\bullet_i$ for $i=1, \ldots, n -p +1$, the operation $\bullet_0$ is also defined in case $p=n+1$. However note that in this case the relation \rmref{lagrandebellezza1} is not fulfilled as the right-hand side would not even be defined.
Further,
similarly to what was said in Remark \ref{infoutili}, one observes that elements in $\cO(0)$, {\em i.e.}, for $p=0$, are allowed to act on elements in $\cM(n)$ for any $n \geq 0$, whereas for elements in $\cM(0)$ only the action by elements in $\cO(0)$ and $\cO(1)$, that is, for $p=\{0,1\}$ has nonzero results.
\end{rem}

The terminology of Definition \ref{piantadiroma} is partially justified by the following proposition
and Remark \ref{narni}:

\begin{prop}
\label{sabaudia}
Let $(\cO, \mu)$ be an operad with multiplication. 
A cyclic
unital comp module over $\cO$ defines a cyclic $k$-module with cyclic operator $t: \cM(n) \to \cM(n)$, along with faces $d_i: \cM(n) \to \cM({n-1})$ and degeneracies $s_j: \cM(n) \to \cM({n+1})$ of the underlying simplicial object given by 
\begin{equation}
\label{colleoppio}
\begin{array}{rcll}
d_i(x) & = & \mu \bullet_{i} x, & i = 0, \ldots, n-1, \\
d_n(x) & = & \mu \bullet_0 t (x), & \\
s_j (x)& = & e \bullet_{j+1} x, & j = 0, \ldots, n, \\
\end{array}
\end{equation}
where $x \in \cM(n)$. 
\end{prop}

\begin{proof}
The proof amounts in a straightforward verification of the simplicial and cyclic identities that define a cyclic module (see, {\em e.g.}, \cite[\S2.5.1]{Lod:CH}); we only prove the (somewhat) less obvious ones. For example, that $s_is_j = s_{j+1}s_i$ for $i \leq j$ and $d_i d_j = d_{j-1}d_i$ in case $i < j-1 < n-1$ follows directly from \rmref{SchlesischeStr}, whereas the case $j = i+1 < n$ needs the property \rmref{distinguished1}:
$$
d_i d_{i+1}(x) = \mu \bullet_i ( \mu \bullet_{i+1} x) = (\mu \circ_2 \mu) \bullet_{i} x =  (\mu \circ_1 \mu) \bullet_{i} x 
=  \mu \bullet_i ( \mu \bullet_{i} x) = d_i d_i (x).
$$
The case $i = n-1, j = n$ needs not only \rmref{distinguished1} but also \rmref{lagrandebellezza1} and \rmref{lagrandebellezza2}:
\begin{equation*}
\begin{split}
d_{n-1}  d_n(x) 
&= \mu \bullet_0 t( \mu \bullet_0 t(x))  
= \mu \bullet_0 (\mu \bullet_1 t^2(x))
=   (\mu \circ_2 \mu) \bullet_0 t^2(x) \\
& =   (\mu \circ_1 \mu) \bullet_0 t^2(x)
= \mu \bullet_0 t t^{n-1} (\mu \bullet_0 t^2(x)) 
= \mu \bullet_0 t (\mu \bullet_{n-1} t^{n+1}(x)) \\
& = \mu \bullet_0 t (\mu \bullet_{n-1} x)
= d_{n-1} d_{n-1}(x). 
\end{split}
\end{equation*}
Likewise, we prove the compatibility between faces and degeneracies, as an illustration why \rmref{distinguished2} and the unitality \rmref{SchlesischeStr-1} are needed:
$$
d_i s_i(x) = \mu \bullet_i (e \bullet_{i+1} x) =  (\mu \circ_i e) \bullet_{i} x = \mathbb{1} \bullet_i x = x.
$$
In the same spirit, one proves the identities for a cyclic module. We only compute, using \rmref{lagrandebellezza1} and \rmref{lagrandebellezza2},
$$
s_0 t(x) = e \bullet_1 t(x) = t^{n+2} (e \bullet_1 t(x) )  = t^{2} (e \bullet_{n+1} t^{n+1}(x) )  = t^{2} (e \bullet_{n+1} x ) = t^2 s_n(x),  
$$
and leave the rest of the simplicial and cyclic identities to the reader. 
\end{proof}

When referring to $\cM$ as a simplicial $k$-module, we denote it by $\cM_\bull$, where $\cM_n := \cM(n)$, and shall not distinguish in notation between a simplicial module and its associated chain complex.

As usual, one defines the {\em Hochschild} or {\em simplicial boundary operator}
$
        \bb := \sum_{i=0}^n (-1)^i \dd_i.  
$
In the situation of Proposition \ref{sabaudia}, we obtain
\begin{equation}
\label{appolloni2}
        \bb := \sum_{i=0}^{n-1} (-1)^i \mu \bullet_i x + (-1)^n \mu \bullet_0 t(x),  
  \end{equation}
and, as on every para-cyclic $k$-module, we introduce the {\em norm operator}, 
the {\em extra  degeneracy},
and the \emph{cyclic differential} 
\begin{equation*}
\label{extra}
        N := \sum_{i=0}^n (-1)^{in} \ttt^i, \qquad         
        \sss_{-1} := \ttt \, \sss_n,\qquad
        \BB=(\mathrm{id}-\ttt) \, \sss_{-1} \, N,
\end{equation*}
such that $(\cM_\bull, b, B)$ forms a mixed complex. Correspondingly, we define:

\begin{definition}
\label{avviso}
For any cyclic unital comp module $\cM$ over an operad $\cO$ with multiplication, 
we call the homology of $(\cM_\bull, b)$, denoted by $H_\bull(\cM)$, its ({\em simplicial} or {\em Hochschild}) {\em homology}, and the homology of the mixed complex $(\cM_\bull, b, B)$, denoted $HC_\bull(\cM)$, the {\em cyclic homology} of $\cM$.
\end{definition}

\begin{rem}
\label{normalised}
Most of the time, we will work on the normalised complex $\bar{\cM}_\bull$
of $\cM_\bull$, meaning the quotient by the subcomplex spanned by the
images of the degeneracy maps of this simplicial 
$k$-module, that is, given by the cokernel of the degeneracy maps \hbox{$s_j := e \bullet_{j+1} -$} for $j = 0, \ldots, n$. Remember that $\BB$ coincides on the 
normalised complex 
$\bar{\cM}_\bull$ with the map (induced by) $\sss_{-1} \, N$, so we 
take the liberty to denote the latter by 
$\BB$ as well, as we, in what follows, will only consider the induced map on the
normalised complex. Similarly, $\bar{\cO}^\bull$ denotes the intersection of the kernels of the codegeneracies in the cosimplicial $k$-module $\cO^\bull$.
\end{rem}

\begin{rem}
\label{narni}
We did not call in a purely random manner the action $\bullet_0$ an ``extra'' comp module map: in view of \rmref{lagrandebellezza1}, \rmref{lagrandebellezza2}, and \rmref{colleoppio}, one computes
\begin{equation*}
\label{extra1}
e \bullet_0 x = t^{n+2}(e \bullet_0 x)  = t (e \bullet_{n+1} t^{n+1}(x))    = t (e \bullet_{n+1} x) = t s_n (x) = s_{-1}(x). 
\end{equation*}
for $x \in \cM(n)$, hence the extra degeneracy associated to the cyclic module. In this sense, the definition of the degeneracies in \rmref{colleoppio} even extends to $j = -1$. Written this way, we can express $B = s_{-1} N$ on the normalised complex as follows:
\begin{equation}
\label{extra2}
B(x) = \sum^n_{i=0} (-1)^{in} e \bullet_0 t^i(x), 
\end{equation}
where $x \in \bar{\cM}(n)$.
\end{rem}

\section{The Gerstenhaber module structure on cyclic comp modules}

In the rest of this article, we will {\em always} assume that $\cO$ be an operad with multiplication $\mu$ and $\cM$ a cyclic unital comp module over $\cO$. 

\subsection{The cap product and the Lie derivative}
\label{federtasche}
In this subsection, we define the operadic generalisation of the cap product and the Lie derivative known in various fields in mathematics, for example, in classical differential geometry of Lie algebroids or Lie-Rinehart algebras (see, for example, \cite{Car:LTDUGDLEDUEFP, Rin:DFOGCA}), or in the Hochschild theory of associative algebras (see, {\em e.g.}, \cite{CarEil:HA, Get:CHFATGMCICH, NesTsy:OTCROAA}). In a certain sense, the cap product and the Lie derivative can be seen as a dualisation of the cup product \rmref{cupco} and the Gerstenhaber bracket \rmref{zugangskarte}.

\begin{definition}
\label{volterra}
Let $(\cO, \mu)$ be an operad with multiplication $\mu \in \cO(2)$ and $\cM$ a cyclic unital comp module over $\cO$. 
\begin{enumerate}
\compactlist{99}
\item
The {\em cap product} 
$$
\iota_\gvf := \gvf \smallfrown \cdot: \cM(n) \to \cM({n-p})
$$ 
of $\varphi \in \cO(p)$ with 
$x \in \cM(n)$ is defined by
\begin{equation}
\label{alles4}
\gvf \smallfrown x := 
(\mu \circ_2 \gvf) \bullet_0 x.  
\end{equation}
\item
The {\em Lie derivative}  
$$
        \lie_\varphi : \cM(n) \rightarrow 
        \cM({n-p+1})
$$
of $x \in \cM(n)$
along $\gvf \in \cO(p)$ with $p < n+1$ is defined to be 
\begin{equation}
\label{messagedenoelauxenfantsdefrance2}
        \lie_\varphi x :=   \!\!      
\sum^{n-p+1}_{i=1} \!\! (-1)^{(p-1)(i-1)} \varphi \bullet_i x 
+ 
\sum^{p}_{i=1} (-1)^{n(i-1) + p - 1} \varphi \bullet_0 \ttt^{i-1} (x).
\end{equation}
In case $p = n+1$, this means that
$$
        \lie_\varphi x := (-1)^{p-1} \varphi \bullet_0  N(x),
$$
and for $p > n+1$, we define 
$\lie_\varphi := 0$. 
\end{enumerate}
\end{definition}

\begin{rem}
If $\gvf \in \cO(0)$ is a zero-cochain, the Lie derivative \rmref{messagedenoelauxenfantsdefrance2} along $\gvf$ on $x \in \cM(n)$ 
has to be read as
$
\lie_\gvf(x) = \sum^{n+1}_{i=1} (-1)^{i-1} \varphi \bullet_i x. 
$
\end{rem}

Compare the formal similarity of the cap product \rmref{alles4} with the cup product in \rmref{cupco}.

\subsection{The dg module structure}

Dual to the fact mentioned in \S\ref{responsabilitacivile} that $(\cO^\bull, \gd, \smallsmile)$ is a dga algebra, one has:

\begin{prop}
\label{estintore}
The triple
$(\cM_{-\bull},\bb,\smallfrown)$ defines a left dg module over 
$(\cO^\bull,\delta,\smallsmile)$, \hbox{that is},
\begin{eqnarray}
\label{mulhouse1}
        \iota_\varphi \, \iota_\psi &=&
        \iota_{\varphi \smallsmile \psi},\\
        \label{mulhouse2}
 \bb \, \iota_\varphi - (-1)^p \iota_\varphi \, \bb =:
        [\bb, \iota_\varphi] 
&=& \iota_{\gd\varphi},
\end{eqnarray} 
where $\gvf \in \cO(p)$, $\psi \in \cO$.
\end{prop} 
\begin{proof}
Eq.~\rmref{mulhouse1} is proven as follows, using \rmref{SchlesischeStr}, \rmref{danton}, and \rmref{distinguished1}:
\begin{equation*}
\begin{split}
\iota_\gvf \iota_\psi x 
&= (\mu \circ_2 \gvf) \bullet_0 (( \mu \circ_2 \psi) \bullet_0 x)  
= ((\mu \circ_2 \gvf) \circ_1 (\mu \circ_2 \psi)) \bullet_0 x  
\\
&
= 
\big(((\mu \circ_2 \gvf) \circ_1 \mu) \circ_2 \psi\big) \bullet_0 x  
=
\big(((\mu \circ_1 \mu) \circ_3 \gvf) \circ_2 \psi\big) \bullet_0 x  
\\
&
= \big(((\mu \circ_2 \mu) \circ_3 \gvf) \circ_2 \psi\big) \bullet_0 x  
= \big((\mu \circ_2 (\mu \circ_2 \gvf)) \circ_2 \psi\big) \bullet_0 x  
\\
&
= \big(\mu \circ_2 ((\mu \circ_2 \gvf) \circ_1 \psi)\big) \bullet_0 x  
= \iota_{\gvf \smallsmile \psi} x.
\end{split}
\end{equation*}
As for \rmref{mulhouse2}, we compute for $x \in \cM(n)$ with the the help of \rmref{SchlesischeStr}, \rmref{distinguished1}, 
\rmref{danton}, and \rmref{lagrandebellezza1} for the simplicial pieces that appear in the commutator $[\bb,\iota_\varphi]$:
\begin{equation*}
\begin{split}
d_0\iota_\gvf x 
&= \mu \bullet_0 (( \mu \circ_2 \gvf) \bullet_0 x)  
= (\mu \circ_1 (\mu \circ_2 \gvf)) \bullet_0 x  \\
&= ((\mu \circ_1 \mu) \circ_2 \gvf) \bullet_0 x  
= ((\mu \circ_2 \mu) \circ_2 \gvf) \bullet_0 x  
= (\mu \circ_2 (\mu \circ_1 \gvf)) \bullet_0 x,  
\end{split}
\end{equation*}
along with
$
d_i\iota_\gvf x 
= \mu \bullet_i (( \mu \circ_2 \gvf) \bullet_0 x)  
$
for $i = 1, \ldots, n-p-1$,
and
\begin{equation*}
\begin{split}
d_{n-p}\iota_\gvf x 
&= \mu \bullet_0 t(( \mu \circ_2 \gvf) \bullet_0 x)  
= \mu \bullet_0 (( \mu \circ_2 \gvf) \bullet_1 t(x))  
= (\mu \circ_2 (\mu \circ_2 \gvf)) \bullet_0 t(x).
\end{split}
\end{equation*}
On the other hand,
\begin{equation*}
\begin{split}
\iota_\gvf d_0 x 
&= 
\mu \bullet_0 (\gvf \bullet_1 ( \mu \bullet_0 x))  
= \mu \bullet_0 ( \mu \bullet_0 (\gvf \bullet_2 x))  
= (\mu \circ_1 \mu) \bullet_0 (\gvf \bullet_2 x) \\  
&= ((\mu \circ_1 \mu) \circ_3 \gvf) \bullet_0 x  
= ((\mu \circ_2 \mu) \circ_3 \gvf) \bullet_0 x  
= (\mu \circ_2 (\mu \circ_2 \gvf)) \bullet_0 x,  
\end{split}
\end{equation*}
along with 
$$
\iota_\gvf d_i x 
= 
\mu \bullet_0 ((\gvf \bullet_i  \mu) \bullet_1 x)  
= (\mu \circ_2 (\gvf \circ_i \mu)) \bullet_0 x   
$$
for $i=1, \ldots, p$, whereas for $i= p+1, \ldots, n-1$ one has
$$
\iota_\gvf d_i x 
= (\mu \circ_2 \gvf) \bullet_0 (\mu \bullet_i x) = \mu \bullet_{i-p} ((\mu \circ_2 \gvf) \bullet_0 x) =   \mu \bullet_{k} ((\mu \circ_2 \gvf) \bullet_0 x),  
$$
where now $k = 1, \ldots, n-p-1$. Finally,
\begin{equation*}
\begin{split}
\iota_\gvf d_n x 
&= 
(\mu \circ_2 \gvf) \bullet_0 (\mu \bullet_0 t x)
= ((\mu \circ_2 \gvf) \circ_1 \mu) \bullet_0 t(x)  
= ((\mu \circ_1 \mu) \circ_3 \gvf) \bullet_0 t(x) \\
&
 = ((\mu \circ_2 \mu) \circ_3 \gvf) \bullet_0 t(x)
 = (\mu \circ_2 (\mu \circ_2 \gvf)) \bullet_0 t(x).
\end{split}
\end{equation*}
Hence,
\begin{small}
\begin{equation*}
\begin{split}
b&\iota_\gvf x - (-1)^p \iota_\varphi \, \bb x \\
&= 
(\mu \circ_2 (\mu \circ_1 \gvf)) \bullet_0 x 
+ (-1)^{p-1} (\mu \circ_2 (\mu \circ_2 \gvf)) \bullet_0 x 
+ (-1)^{p+i-1} (\mu \circ_2 (\gvf \circ_i \mu)) \bullet_0 x \\
& = \iota_{\{\mu, \gvf\}} x = \iota_{\gd\gvf} x,
\end{split}
\end{equation*}
\end{small}
which concludes the proof of the proposition.
\end{proof}

\subsection{The dg Lie algebra module structure}

In this section, we prove that 
the Lie derivative $\lie$ defines a dg 
Lie algebra representation of $(\cO^\bull[1], \{.,.\})$ on $\cM_{-\bull}$: 

\begin{theorem}
\label{feinefuellhaltertinte}
Let $(\cO, \mu)$ be an operad with multiplication and $\cM$ a cyclic unital comp module over $\cO$.
For any two $\gvf \in \cO(p)$ 
and $\psi \in \cO(q)$, 
we have
\begin{equation}
\label{weimar}
[\lie_\varphi, \lie_\psi] = \lie_{\{\varphi, \psi\}},
\end{equation}
where the bracket on the right-hand side is the Gerstenhaber bracket
\rmref{zugangskarte}.
Furthermore, the simplicial differential \rmref{appolloni2} on $\cM_\bull$ is given (up to a sign) by the Lie derivative along the operad multiplication, that is,
\begin{equation}
\label{sachengibts}
        \bb = - \lie_\mu, 
\end{equation}
and therefore
\begin{equation}
\label{alles1}
[\bb, \lie_\varphi] + \lie_{\gd\varphi}= 0. 
\end{equation}
\end{theorem}

\begin{proof}
The statement is proven by an explicit computation using \rmref{SchlesischeStr}, \rmref{lagrandebellezza1}, and \rmref{lagrandebellezza2}, plus some enforced double sum yoga, and is essentially straightforward.
 
Assume that $1 \leq q \leq p$ and 
$p+q \leq n+1$; the proof for zero cochains and for the extreme cases $q = 0$,
$p = n+1$ can be carried out by similar, but easier
computations. 
To obtain shorter expressions in what follows, introduce the sign functions
\begin{equation}
\label{nerv2}
\zeta(p,i):= (p-1)(i-1), \qquad \xi(n,p,i) := n(i-1) + p-1.
\end{equation}
Then, for computing the left-hand side in \rmref{weimar}, consider
for $x \in \cM(n)$
\begin{footnotesize}
\begin{equation*}
\begin{split}
\lie_\varphi \lie_\psi x 
&= {\sum^{n-p-q+2}_{i=1} \, \sum^{n-q+1}_{j=1} (-1)^{\zeta(p,i) + \zeta(q,j)} \gvf \bullet_i (\psi \bullet_j x)} 
\\
&\qquad
+ 
{\sum^{n-p-q+2}_{i=1} \, \sum^{q}_{j=1} (-1)^{\zeta(p,i) + \xi(n, {q,j})} 
\gvf \bullet_i (\psi \bullet_0 \ttt^{j-1}(x))
}  
\\
&\qquad 
+ {\sum^{p}_{i=1} \, \sum^{n-q+1}_{j=1} (-1)^{\xi({n-q+1},{p,i}) + \zeta(q,j)} 
\gvf \bullet_0 \ttt^{i-1} (\psi \bullet_j x)}  
\\
&\qquad 
+ {\sum^{p}_{i=1} \, \sum^{q}_{j=1} (-1)^{\xi({n-q+1}, {p,i}) + \xi({n}, {q,j})} 
\gvf \bullet_0 \ttt^{i-1} (\psi \bullet_0  \ttt^{j-1}(x))} 
\\
&=: (1) + (2) + (3) + (4) ,  
\end{split}
\end{equation*}
\end{footnotesize}
where the numbers $(1), (2), (3), (4)$ are meant to correspond in the same order to the respective sums in the equation before. Likewise,
\begin{footnotesize}
\begin{equation*}
\begin{split}
-(-1)^{(p-1)(q-1)} \lie_\psi \lie_\varphi  
&=  {-
\sum^{n-p-q+2}_{j=1} \, \sum^{n-p+1}_{i=1} (-1)^{\zeta(q,j) + \zeta(p,{i + q-1})} 
\psi \bullet_j (\gvf \bullet_i x)
} 
\\
&
\quad {-  \sum^{n-q-p+2}_{j=1} \, \sum^{p}_{i=1} (-1)^{\zeta(q,{j+p-1}) + \xi(n,{p,i})}  
\psi \bullet_j (\gvf \bullet_0 \ttt^{i-1}(x))
}
\\
& 
\quad
 {-\sum^{q}_{j=1} \, \sum^{n-p+1}_{i=1} (-1)^{\xi(n,{q,j}) + \zeta(p,{i+j-1})}  
\psi \bullet_0 \ttt^{j-1} (\gvf \bullet_i x)}  
\\
& 
\quad
{- \sum^{q}_{j=1} \, \sum^{p}_{i=1} (-1)^{\xi(n, {p+q-1,i+j-1}) + \zeta(p, {j+q-1})}  
\psi \bullet_0 \ttt^{j-1} (\gvf \bullet_0  \ttt^{i-1}(x))
} \\
&=: (5) + (6) + (7) + (8).  
\end{split}
\end{equation*}
\end{footnotesize}
On the other hand, the right-hand side of \rmref{weimar} reads, by accordingly applying \rmref{SchlesischeStr}:
\begin{footnotesize}
\begin{equation*}
\begin{split}
\lie_{\{\varphi, \psi\}}  &= 
\sum^{n-p-q+2}_{i=1} \, \sum^{i+p-1}_{j=i} (-1)^{\zeta(p,i) + \zeta(q,j)} 
(\gvf \circ_{j-i+1} \psi) \bullet_i x
\\
&\quad
 + \sum^{p+q-1}_{i=1} \, \sum^{p}_{j=1} (-1)^{\xi(n,{p,i}) + \zeta(q,{j-1})}  
\gvf \bullet_0 (\psi \bullet_{j-1} \ttt^{i-1}(x))
\\
&\quad 
- \!\! \sum^{n-p-q+2}_{i=1} \, \sum^{i+q-1}_{j=i} (-1)^{\zeta(q,j) + \zeta(p, {i+q-1})} 
(\psi \circ_{j-i+1} \gvf) \bullet_i x
\\
&\quad
  - \sum^{p+q-1}_{i=1} \, \sum^{q}_{j=1} (-1)^{\xi(n,{q,j}) + \zeta(p,{q+i})}  
\psi \bullet_0 (\gvf \bullet_{j-1} \ttt^{i-1}(x))
\\
&=: (9) + (10) + (11) + (12).  
\end{split}
\end{equation*}
\end{footnotesize}
We now decompose the sum (1) into the three intervals of the index $i$ that appear in the property \rmref{SchlesischeStr} and subsequently use the relations in \rmref{SchlesischeStr}:
\begin{footnotesize}
\begin{equation}
\label{weimar2}
\begin{split}
& (1) = {\sum^{n-q+1}_{j=p+1} \, \sum^{j-p}_{i=1} (-1)^{\zeta(p,i) + \zeta(q,j)} \psi \bullet_{j-p+1} (\gvf \bullet_i x)} 
\\
&
\qquad
+ 
 \sum^{n-q+1}_{j=1} \, \sum^{j}_{i=j-p+1} (-1)^{\zeta(p,i) + \zeta(q,j)} (\gvf \circ_{j-i+1} \psi) \bullet_i x 
\\
&\qquad
+
 {\sum^{n-q-p+1}_{j=1} \, \sum^{n-p-q+2}_{i=j+1} (-1)^{\zeta(p,i) + \zeta(q,j)} \psi \bullet_j (\gvf \bullet_{i+q-1} x)} \\
&= 
{\sum^{n-q-p+1}_{i=1} \, \sum^{n-q-p+2}_{j=i+1} (-1)^{\zeta(p,i) + \zeta(q, {j-p+1})} \psi \bullet_{j} (\gvf \bullet_i x)} 
\\
&
\qquad
+ 
 \sum^{n-q-p+2}_{i=1} \, \sum^{i+p-1}_{j=i} (-1)^{\zeta(p,i) + \zeta(q,j)} (\gvf \circ_{j-i+1} \psi) \bullet_i x 
\\
&\qquad
+
 {\sum^{n-q-p+1}_{j=1} \, \sum^{n-p+1}_{i=j+q} (-1)^{\zeta(p,{i+q-1}) + \zeta(q,j)} \psi \bullet_j (\gvf \bullet_{i} x)} 
=: (13) + (14) + (15), 
\end{split}
\end{equation}
\end{footnotesize}
where we applied the double sum transformation
$$
\sum^{n-q+1}_{p+1} \, \sum^{j-p}_{i=1} = \sum^{n-q-p+1}_{i=1} \, \sum^{n-q+1}_{j=i+p} 
$$
to the first sum in the second step, and a similar transformation
to the second sum. Now one observes that (14) = (9) and that the first term (13) and the third term (15) in \rmref{weimar2} cancel with the respective third and first term that would arise from an analogous decomposition of (5), using the property
$\zeta(p,{i+q-1}) + \zeta(q,j) = \zeta(p,i) + \zeta(q,{j+p-1})$ for the sign function $\zeta$ in \rmref{nerv2}, 
whereas the second term in such a decomposition would analogously cancel with (11). Hence the remaining terms to consider are (2), (3), (4), and (6), (7), (8), along with (10) and (12). We are going to show that (3), (4), (6), as well as (10) cancel, and by symmetry this also shows that analogously (2), (7), (8), as well as (12) cancel:
one has, using \rmref{lagrandebellezza1} and \rmref{lagrandebellezza2} along with a double sum transformation as above,
\begin{footnotesize}
\begin{equation*}
\begin{split}
(3) &= 
{\sum^{p}_{i=1} \, \sum^{n-q-i+2}_{j=1} (-1)^{\xi({n-q+1},{p,i}) + \zeta(q,j)} \gvf \bullet_0 t^{i-1}(\psi \bullet_j x)} 
\\
&\qqquad
+ 
{\sum^{p}_{i=2} \, \sum^{n-q+1}_{j=n-q-i+3} (-1)^{\xi({n-q+1},{p,i}) + \zeta(q,j)} \gvf \bullet_0 t^{i-1}(\psi \bullet_j x)} 
\\
&=
{\sum^{p}_{i=1} \, \sum^{n-q-i+2}_{j=1} (-1)^{\xi({n-q+1},{p,i}) + \zeta(q,j)} \gvf \bullet_0 (\psi \bullet_{j+i-1} t^{i-1}(x))} 
\\
&\qqquad
+ 
{\sum^{p}_{i=2} \, \sum^{n-q+1}_{j=n-q-i+3} (-1)^{\xi({n-q+1},{p,i}) + \zeta(q,j)} \gvf \bullet_0 t^{i-1}t^{n-q-i+3}(\psi \bullet_{j-n+i+q+1} t^{i+q}(x))} 
\\
&=
{\sum^{p}_{i=1} \, \sum^{n-q+1}_{j=i} (-1)^{\xi({n},{p,i}) + \zeta(q,j)} \gvf \bullet_0 (\psi \bullet_{j} t^{i-1}(x))} 
\\
&
\qquad
+ 
{\sum^{p}_{i=2} \, \sum^{i-2}_{j=0} (-1)^{\xi({n},{p,i}) + \zeta(q,{j-1})} \gvf \bullet_0 (\psi \bullet_j t^{i+q-2}(x))} 
\\
&=
{\sum^{p}_{i=1} \, \sum^{n-q+1}_{j=i} (-1)^{\xi({n},{p,i}) + \zeta(q,j)} \gvf \bullet_0 (\psi \bullet_{j} t^{i-1}(x))} 
\\
&\qquad
+ 
{\sum^{p-1}_{j=1} \, \sum^{p+q-1}_{i=q+j} (-1)^{\xi({n},{p,i}) + \zeta(q,{j-1})} \gvf \bullet_0 (\psi \bullet_{j-1} t^{i-1}(x))} 
\\
&=: (3a) + (3b). 
\end{split}
\end{equation*}
\end{footnotesize}
Moreover, by \rmref{lagrandebellezza1} again,
\begin{footnotesize}
\begin{equation*}
\begin{split}
(4) &= \sum^{p}_{i=1} \, \sum^{q}_{j=1} (-1)^{\xi({n-q+1},{p,i}) + \xi({n},{q, j})} \gvf \bullet_0 (\psi \bullet_{i-1} t^{i+j-2}(x)) 
\\
&
= \sum^{p}_{j=1} \, \sum^{j+q-1}_{i=j} (-1)^{\xi({n},{p,i}) + \zeta(q, {j-1})} \gvf \bullet_0 (\psi \bullet_{j-1} t^{i-1}(x)),
\end{split}
\end{equation*}
\end{footnotesize}
along with
\begin{footnotesize}
\begin{equation*}
\begin{split}
(6) &= - \!\!\! \sum^{n-p-q+2}_{j=1} \, \sum^{p}_{i=1} (-1)^{\xi({n},{p,i}) + \zeta(q,{j+p-1})} \gvf \bullet_0 (\psi \bullet_{j+p-1} t^{i-1}(x)) \\
&= - \!\!\! \sum^{n-q+1}_{j=p} \, \sum^{p}_{i=1} (-1)^{\xi({n},{p,i}) + \zeta(q,{j})} \gvf \bullet_0 (\psi \bullet_{j} t^{i-1}(x)),
\end{split}
\end{equation*}
\end{footnotesize}
using \rmref{SchlesischeStr} again, whereas
\begin{footnotesize}
\begin{equation*}
\begin{split}
(10) &= \sum^{p-1}_{j=1} \, \sum^{p+q-1}_{i=1} (-1)^{\xi({n},{p,i}) + \zeta(q,{j-1})} \gvf \bullet_0 (\psi \bullet_{j-1} t^{i-1}(x)) 
\\
&\qquad
+  \sum^{p+q-1}_{i=1} (-1)^{\xi({n},{p,i}) + \zeta(q,{p-1})} \gvf \bullet_0 (\psi \bullet_{p-1} t^{i-1}(x)) 
=:(10a) + (10b).
\end{split}
\end{equation*}
\end{footnotesize}
One then computes, again with the help of the usual double sum yoga,
\begin{footnotesize}
\begin{equation*}
\label{weimar4}
\begin{split}
& \Big[\big[[(3b) - (10a)] - (10b)\big] +  (4)\Big] + \big[(3a) + (6)\big] 
\\
&\quad = 
\Big[\big[- \sum^{p-1}_{j=1} \, \sum^{j+q-1}_{i=1}  (-1)^{\xi({n},{p,i}) + \zeta(q,{j-1})} \gvf \bullet_0 (\psi \bullet_{j-1} t^{i-1}(x)) - (10b)\big] + (4) \Big] 
\\
&\qqquad
+ \sum^{p-1}_{i=1} \, \sum^{p-1}_{j=i}  (-1)^{\xi({n},{p,i}) + \zeta(q,{j})} \gvf \bullet_0 (\psi \bullet_{j} t^{i-1}(x))
\end{split}
\end{equation*}
\begin{equation*}
\begin{split}
&\quad 
= 
\Big[- \sum^{p}_{j=1} \, \sum^{j+q-1}_{i=1}  (-1)^{\xi({n},{p,i}) + \zeta(q,{j-1})} \gvf \bullet_0 (\psi \bullet_{j-1} t^{i-1}(x)) + (4) \Big] 
\\
&\qqquad
+ \sum^{p-1}_{i=1} \, \sum^{p}_{j=i+1}  (-1)^{\xi({n},{p,i}) + \zeta(q,{j-1})} \gvf \bullet_0 (\psi \bullet_{j-1} t^{i-1}(x))
\\
&= 
- \!\! \sum^{p}_{j=2} \, \sum^{j-1}_{i=1}  (-1)^{\xi({n},{p,i}) + \zeta(q,{j-1})} \gvf \bullet_0 (\psi \bullet_{j-1} t^{i-1}(x))
\\
&\qquad 
+ \sum^{p-1}_{i=1} \, \sum^{p}_{j=i+1}  (-1)^{\xi({n},{p,i}) + \zeta(q,{j-1})} \gvf \bullet_0 (\psi \bullet_{j-1} t^{i-1}(x)) = 0,
\end{split}
\end{equation*}
\end{footnotesize}
and this concludes the proof of Eq.~\rmref{weimar}.

Finally, Eq.~\rmref{sachengibts} directly follows from \rmref{appolloni2}, 
and \rmref{alles1} then results from this along with \rmref{weimar} and \rmref{erfurt}.
\end{proof}

\subsection{The Gerstenhaber module structure}

By \rmref{mulhouse2} and \rmref{alles1}, both
operators $\iota_\varphi$ and $\lie_\varphi$ 
descend to well defined operators on 
the simplicial homology 
$H_\bull(\cM)$ as soon as $\varphi$ is a cocycle. 
In this case, the following theorem together with Theorem \ref{feinefuellhaltertinte} proves that 
$\iota$ and $\lie$ turn $H_\bull(\cM)$ 
into a module over the 
Gerstenhaber algebra $H^\bull(\cO)$ in the sense of Definition \ref{golfoaranci}:

\begin{theorem}
\label{waterbasedvarnish}
Let $(\cO,\mu)$ be an operad with multiplication and $\cM$ a cyclic unital comp module over $\cO$. For any two cocycles 
$\varphi \in \cO(p)$, $\psi \in \cO(q)$, 
the induced maps 
$$
\lie_\varphi:   H_\bull(\cM) \to H_{\bull-p+1}(\cM), \qquad
\iota_\psi:  H_\bull(\cM) \to H_{\bull-q}(\cM) 
$$
satisfy
\begin{equation}
\label{radicale1}
[\iota_\psi, \lie_\varphi] = \iota_{\{\psi, \gvf\}}.
\end{equation}
\end{theorem}

\begin{proof}
Assume without loss of generality that $q \geq  p$ and $p + q \leq n+1$, and recall the sign functions $\zeta$ and $\xi$ from \rmref{nerv2}.
A direct computation gives
\begin{footnotesize}
\begin{equation*}
\begin{split}
\iota_\psi \lie_\gvf(x) = 
 \sum^{n-p+1}_{i=1} (-1)^{\zeta(p,i)} (\mu \circ_2 \psi) \bullet_0 (\gvf  \bullet_i x) 
+ \sum^{p}_{i=1} (-1)^{\xi(n,{p,i})} ((\mu \circ_2 \psi) \circ_1 \gvf) \bullet_0 t^{i-1}(x), 
\end{split}
\end{equation*}
\end{footnotesize}
whereas
\begin{footnotesize}
\begin{equation*}
\begin{split}
- (-1)^{q(p-1)} \lie_\gvf \iota_\psi(x) 
&= - \!\! \sum^{n-p-q+1}_{i=1} (-1)^{\zeta(p,{i+q-1})} ( \mu \circ_2 \psi) \bullet_0 (\gvf \bullet_{i+q} x) 
\\
& \qquad
\, - \!\!  \sum^{p}_{i=1} (-1)^{\xi({n-q},{p,i}) + \zeta(p,{q-1})} (\gvf \circ_i (\mu \circ_2 \psi)) \bullet_0 t^{i-1}(x). 
\end{split}
\end{equation*}
\end{footnotesize}
Hence, using the expression \rmref{cupco} for the cup product, we have
\begin{footnotesize}
\begin{equation}
\label{nichschonwieder}
\begin{split}
[\iota_\psi, \lie_\gvf](x)  
&= \sum^{q}_{i=1} (-1)^{\zeta(p,i)} (\mu \circ_2 \psi) \bullet_0(\gvf  \bullet_i x)
\\
&
\qquad 
- \sum^{p}_{i=1} (-1)^{\xi({n-q},{p,i}) + \zeta(p,{q-1})} (\gvf \circ_i (\mu \circ_2 \psi)) \bullet_0 t^{i-1}(x) \\
& \qquad + \sum^{p}_{i=1} (-1)^{\xi(n,{p,i})} (\psi \smallsmile \gvf) \bullet_0 t^{i-1}(x). 
\end{split}
\end{equation}
\end{footnotesize}
The first summand in \rmref{nichschonwieder} simplifies to 
\begin{footnotesize}
\begin{equation*}
\begin{split}
 \sum^{q}_{i=1} (-1)^{\zeta(p,i)} (\mu \circ_2 \psi) \bullet_0(\gvf  \bullet_i x) 
&= \sum^{q}_{i=1} (-1)^{\zeta(p,i)} (\mu \circ_2 (\psi \circ_{i} \gvf)) \bullet_0 x 
= \iota_{\psi \bar\circ \gvf} x, 
\end{split}
\end{equation*}
\end{footnotesize}
hence we are left to show that 
\begin{footnotesize}
\begin{equation}
\label{lucca}
\begin{split}
&
\sum^{p}_{i=1} (-1)^{\xi(n,{p,i})} (\psi \smallsmile \gvf) \bullet_0 t^{i-1}(x) 
\, - 
\!\!
\sum^{p}_{i=1} (-1)^{\xi({n-q},{p,i}) + \zeta(p,{q-1})} (\gvf \circ_i (\mu \circ_2 \psi)) \bullet_0 t^{i-1}(x) 
\\
&
 = - (-1)^{\zeta(p,q)} \iota_{\gvf \bar\circ \psi}x. 
\end{split}
\end{equation}
\end{footnotesize}
To prove \rmref{lucca}, we will make use of the following standard result \cite{Ger:TCSOAAR}, but note that we use here the version with the opposite cup product compared to {\em op.~cit.}:

\begin{lem}
Let $(\cO, \mu)$ be an operad with multiplication. For any $\gvf \in \cO(p)$, $\psi \in \cO(q)$ the identity
\begin{equation*}
        (-1)^{p-1}\varphi \bar\circ \gd \psi 
        - (-1)^{p-1}\gd(\varphi \bar\circ \psi) 
        + \gd \varphi \bar\circ \psi = 
        \gvf \smallsmile \psi 
        - (-1)^{pq} \psi \smallsmile \gvf, 
\end{equation*}
holds with respect to the cup product \rmref{cupco} and the differential \rmref{erfurt}.
\end{lem}

This in particular means that 
$
\psi \smallsmile \gvf = (-1)^{pq} \gvf \smallsmile \psi - (-1)^{p(q-1)} \gd(\varphi \bar\circ \psi)
$ 
if $\varphi$ and $\psi$ are cocycles. Inserting this into \rmref{lucca} and using the definition of $\iota$ in \rmref{alles4}, it remains to prove that
\begin{footnotesize}
\begin{equation}
\label{lucca2}
\begin{split}
\!\!\!\!&\!\!\!\! \sum^{p}_{i=1} (-1)^{\xi(n,{p,i}) + pq} (\gvf \smallsmile \psi) \bullet_0 t^{i-1}(x) 
\, - \!\! 
\sum^{p}_{i=1} (-1)^{\xi({n-q},{p,i}) + \zeta(p,{q-1})} (\gvf \circ_i (\mu \circ_2 \psi)) \bullet_0 t^{i-1}(x)
\\
&= 
\sum^{p}_{i=1} \sum^{p}_{j=1} (-1)^{\xi(n,{pq,i}) + \zeta(q,j)} (\mu \circ_1 (\gvf \circ_j \psi)) \bullet_0 t^{i-1}(x)
\\
&
\qquad
+
\sum^{p}_{i=2} \sum^{p}_{j=1} (-1)^{\zeta(q,{j+p-1}) + \zeta({n-1},i)} (\mu \circ_2 (\gvf \circ_j \psi)) \bullet_0 t^{i-1}(x)
\\
&
\qquad
- \! \sum^{p}_{i=1} \sum^{p}_{j=1} (-1)^{\zeta(q,{j+p-1}) + \zeta({n-1},i)} ((\gvf \circ_j \psi) \bar\circ \mu) \bullet_0 t^{i-1}(x) =: (1) + (2) + (3)
\end{split}
\end{equation}
\end{footnotesize}
holds, where, similarly as in the proof of Theorem \ref{feinefuellhaltertinte} above, the numbers $(1), (2), (3)$ are meant to correspond to the respective three sums in the same order in the equation before; we continue by dealing with three terms.

Since we descended to homology, by means of \rmref{appolloni2} and \rmref{colleoppio}, 
along with \rmref{danton}, \rmref{SchlesischeStr}, and \rmref{lagrandebellezza1}, for any $i, j = 1, \ldots, p$ in the first sum (1), one has:
\begin{footnotesize}
\begin{equation*}
\label{lucca3}
\begin{split}
 (&-1)^{\xi(n,{pq,i}) + \zeta(q,j)} (\mu \circ_1 (\gvf \circ_j \psi)) \bullet_0 t^{i-1}(x) = 
(-1)^{\xi(n,{pq,i}) + \zeta(q,j)} \mu \bullet_0 \big((\gvf \circ_j \psi) \bullet_0 t^{i-1}(x)\big) \\
&= 
(-1)^{\zeta(q,{j+p-1}) + ni} \mu \bullet_0 t \big((\gvf \circ_j \psi) \bullet_0 t^{i-1}(x)\big) 
\\
&
\qquad
+ 
\sum^{n-p-q+1}_{k=1} (-1)^{\xi(n,{pq,i}) + \zeta(q,j) + k-1} \mu \bullet_k \big((\gvf \circ_j \psi) \bullet_0 t^{i-1}(x)\big) \\
&= 
(-1)^{\zeta(q,{j+p-1}) + ni-1} (\mu \circ_2 (\gvf \circ_j \psi)) \bullet_0 t^{i}(x) 
\\
&
\qquad
+ 
\sum^{n-i}_{k=p+q-i} (-1)^{\xi({n-1},{q,i}) + \zeta(q,{j+p-1}) + k-1} (\gvf \circ_j \psi) \bullet_0 t^{i-1}(\mu \bullet_k x) \\
&=: (4) + (5).
\end{split}
\end{equation*}
\end{footnotesize}
By the same arguments, {\em i.e.}, since $x\in H_\bull(\cM)$ and by \rmref{danton}, 
\rmref{SchlesischeStr}, as well as \rmref{lagrandebellezza1}, we continue, using furthermore the cyclicity \rmref{lagrandebellezza2}, by writing
\begin{footnotesize}
\begin{equation*}
\begin{split}
&(5) = 
\sum^{n-i}_{k=p+q-i} (-1)^{\xi({n-1},{q,i}) + \zeta(q,{j+p-1}) + k-1} (\gvf \circ_j \psi) \bullet_0 t^{i-1}(\mu \bullet_k x) \\
&= \sum^{p+q-i-1}_{k=0} (-1)^{\xi({n-1},{q,i}) + \zeta(q,{j+p-1}) + k} (\gvf \circ_j \psi) \bullet_0 t^{i-1}(\mu \bullet_k x) 
\\
&\quad
+ (-1)^{(n-1)i + \zeta(q,{j+p-1})} (\gvf \circ_j \psi) \bullet_0 t^{i-1}(\mu \bullet_0 t(x)) 
\\
&
\qquad
+ \!\! \sum^{n-1}_{k=n-i+1} (-1)^{\xi({n-1},{q,i}) + \zeta(q,{j+p-1}) + k} (\gvf \circ_j \psi) \bullet_0 t^{i-1}(\mu \bullet_k x) 
\\
&= 
\sum^{p+q-2}_{k=i-1} (-1)^{\xi({n},{q,i}) + \zeta(q,{j+p-1}) + k} ((\gvf \circ_j \psi) \circ_{k+1} \mu) \bullet_0 t^{i-1}(x) 
\\
&
\qquad
+ (-1)^{(n-1)i + \zeta(q,{j+p-1})} ((\gvf \circ_j \psi) \circ_i \mu) \bullet_0 t^i(x) 
\\
&\quad  
+ \sum^{n-1}_{k=n-i+1} (-1)^{\xi({n-1},{q,i}) + \zeta(q,{j+p-1}) + k} (\gvf \circ_j \psi) \bullet_0 t^{i-1}t^{n-i+1}(\mu \bullet_{k-n+i-1} t^i(x)) 
\end{split}
\end{equation*}
\begin{equation}
\label{lucca4}
\begin{split}
&= 
\sum^{p+q-1}_{k=i} (-1)^{\xi({n},{q,i}) + \zeta(q,{j+p-1}) + k} ((\gvf \circ_j \psi) \circ_{k} \mu) \bullet_0 t^{i-1}(x) 
\\
&
\qquad
+ \sum^{i}_{k=1} (-1)^{\zeta(q,{j+p-1}) + ni + k} ((\gvf \circ_j \psi) \circ_k \mu) \bullet_0 t^i(x) 
\end{split}
\end{equation}
\end{footnotesize}
for all $i, j = 1, \ldots, p$. Summing over all $i, j$, one sees that 
\begin{footnotesize}
\begin{equation}
\label{lucca5}
\begin{split}
\rmref{lucca4} &= 
\sum^{p}_{i=1} \sum^{p}_{j=1} 
\sum^{p+q-1}_{k=1} (-1)^{\xi({n},{k,i}) + \zeta(q,{j+p-1})} ((\gvf \circ_j \psi) \circ_{k} \mu) \bullet_0 t^{i-1}(x)
\\
&\qquad 
+ \sum^{p}_{j=1} \sum^p_{k=1}(-1)^{\zeta(q,{j+p-1}) + np + k} ((\gvf \circ_j \psi) \circ_k \mu) \bullet_0 t^p(x). 
\end{split}
\end{equation}
\end{footnotesize}
We observe that the first summand in \rmref{lucca5} cancels with (3). 
On the other hand, taking the sum over $i, j$ for the term (4) and adding this to (2) gives:
\begin{footnotesize}
\begin{equation*}
\label{lucca6}
\begin{split}
&\sum^{p}_{i=1} \sum^{p}_{j=1} 
(-1)^{\zeta(q,{j+p-1}) + ni} (\mu \circ_2 (\gvf \circ_j \psi)) \bullet_0 t^{i}(x) + (2) 
\\
&
\qquad
=  
- \sum^{p}_{j=1} (-1)^{\zeta(q,{j+p-1}) + np} (\mu \circ_2 (\gvf \circ_j \psi)) \bullet_0 t^p(x), 
\end{split}
\end{equation*}
\end{footnotesize}
hence, in total, we have proven so far:
\begin{footnotesize}
\begin{equation*}
\label{lucca7}
\begin{split}
(1) + (2) +(3) &= 
\sum^{p}_{j=1} \sum^p_{k=1}(-1)^{\zeta(q,{j+p-1}) + np + k} ((\gvf \circ_j \psi) \circ_k \mu) \bullet_0 t^p(x)
\\
&
\qquad
-
\sum^{p}_{j=1} (-1)^{\zeta(q,{j+p-1}) + np} (\mu \circ_2 (\gvf \circ_j \psi)) \bullet_0 t^p(x), 
\end{split}
\end{equation*}
\end{footnotesize}
where the first sum was a left-over from \rmref{lucca5};
therefore, we need to show now, instead of \rmref{lucca2}, that
\begin{footnotesize}
\begin{equation}
\label{lucca8}
\begin{split}
\!\!\!\!\!\!&\!\!\!\!\!\!\sum^{p}_{i=1} (-1)^{\xi({n},{p,i}) + pq} (\gvf \smallsmile \psi) \bullet_0 t^{i-1}(x) 
-
\sum^{p}_{i=1} (-1)^{\xi({n-q},{p,i}) + \zeta(p,{q-1})} (\gvf \circ_i (\mu \circ_2 \psi)) \bullet_0 t^{i-1}(x)
\\
&= 
\sum^{p}_{j=1} \sum^p_{k=1}(-1)^{\zeta(q,{j+p-1}) + np +k} ((\gvf \circ_j \psi) \circ_k \mu) \bullet_0 t^p(x)
\\
&
\qquad
+
\sum^{p}_{j=1} (-1)^{\zeta(q,{j+p-1}) + np} (\mu \circ_2 (\gvf \circ_j \psi)) \bullet_0 t^p(x) 
\end{split}
\end{equation}
\end{footnotesize}
is true. 
As above, by the fact that $x\in H_\bull(\cM)$ along with Eqs.~\rmref{danton}, \rmref{SchlesischeStr}, \rmref{lagrandebellezza1}, and \rmref{lagrandebellezza2}, we compute for the second summand on the right-hand side in \rmref{lucca8} on homology
\begin{footnotesize}
\begin{equation*}
\label{lucca9}
\begin{split}
&\quad- \sum^{p}_{j=1} (-1)^{\zeta(q,{j+p-1}) + np} (\mu \circ_2 (\gvf \circ_j \psi)) \bullet_0 t^p(x) 
\\
&
= - \sum^{p}_{j=1} (-1)^{\zeta(q,{j+p-1}) + np} \mu \bullet_0 t\big((\gvf \circ_j \psi) \bullet_0 t^{p-1}(x)\big) 
\\
&= \sum^{p}_{j=1} (-1)^{\xi({n},{pq-1,p}) + \zeta(q,{j})} \mu \bullet_0 \big((\gvf \circ_j \psi) \bullet_0 t^{p-1}(x)\big)
\\
&
\quad
+  \sum^{p}_{j=1} \sum^{n-p+q-1}_{k=1} (-1)^{\xi({n},{pq,p}) + \zeta(q,{j-1})+k-1} \mu \bullet_k \big((\gvf \circ_j \psi) \bullet_0 t^{p-1}(x)\big) 
\\
&= \sum^{p}_{j=1} (-1)^{\xi({n},{pq-1,p}) + \zeta(q,{j})} (\mu \circ_1 (\gvf \circ_j \psi)) \bullet_0 t^{p-1}(x)
\\
&
\quad
+  \sum^{p}_{j=1} \sum^{n-p}_{k=q} (-1)^{\xi({n},{pq,p}) + \zeta(q,{j-1}) + k-1} (\gvf \circ_j \psi) \bullet_0 t^{p-1}\big(\mu \bullet_k x\big) 
\\
&=: (6) + (7), 
\end{split}
\end{equation*}
\end{footnotesize}
and further
\begin{footnotesize}
\begin{equation*}
\label{lucca10}
\begin{split}
&(7) = 
\sum^{p}_{j=1} \sum^{q-1}_{k=0} (-1)^{\xi({n},{pq,p}) + \zeta(q,{j-1}) + k} (\gvf \circ_j \psi) \bullet_0 t^{p-1}\big(\mu \bullet_k x\big) 
\\
&\qquad  
+ \sum^{p}_{j=1} \sum^{n-1}_{k=n-p+1} (-1)^{\xi({n},{pq,p}) + \zeta(q,{j-1}) + k} (\gvf \circ_j \psi) \bullet_0 t^{p-1}\big(\mu \bullet_k x\big) 
\\
&\qquad 
+ \sum^{p}_{j=1} (-1)^{\zeta(q,{j-1}) + p(n-q)} (\gvf \circ_j \psi) \bullet_0 t^{p-1}\big(\mu \bullet_0t(x)\big) 
\\
&= \sum^{p}_{j=1} \sum^{q+p-1}_{k=p} (-1)^{\xi({n},{p,k}) + \zeta(q,{j+p-1})} ((\gvf \circ_j \psi) \circ_k \mu) \bullet_0 t^{p-1}(x) 
\\
&\qquad 
+ \sum^{p}_{j=1} \sum^{p}_{k=1} (-1)^{\xi({n},{p-1,k}) + \zeta(q,{j+p-1})} ((\gvf \circ_j \psi) \circ_k \mu) \bullet_0 t^{p}(x) =: (8) + (9). 
\end{split}
\end{equation*}
\end{footnotesize}
One observes that (9) cancels with the first summand in \rmref{lucca8}, whereas, on homology again, and since $\psi$ and $\gvf$ are cocycles, we realise that
\begin{footnotesize}
\begin{equation*}
\label{lucca11}
\begin{split}
&(8) = \sum^{p}_{j=1} \sum^{q}_{k=p-j+1} (-1)^{\zeta(n,p) + q(p-j) + k-1} (\gvf \circ_j (\psi \circ_k \mu)) \bullet_0 t^{p-1}(x) 
\\
&\qquad
+ \sum^{p}_{j=1} \sum^{p}_{k=j+1} (-1)^{\xi({n},{k,p}) + \zeta(q,{j+p})} ((\gvf \circ_k \mu) \circ_j \psi) \bullet_0 t^{p-1}(x)
\\
&
 = \sum^{p}_{j=1} \sum^{p-j}_{k=1} (-1)^{\zeta(n,p) + q(p-j) + k} (\gvf \circ_j (\psi \circ_k \mu)) \bullet_0 t^{p-1}(x) 
\\
&
\qquad
+  \sum^{p}_{j=1} (-1)^{\xi({n},{qj,p}) + \zeta(q,{p})} (\gvf \circ_j (\mu \circ_1 \psi)) \bullet_0 t^{p-1}(x) 
\\
& \qquad 
+  \sum^{p}_{j=1} (-1)^{\zeta(n,{p}) + q(p-j)} (\gvf \circ_j (\mu \circ_2 \psi)) \bullet_0 t^{p-1}(x)
\\
&
\qquad
+  \sum^{p}_{j=1} \sum^{j}_{k=1} (-1)^{\xi({n},{k-1,p}) + \zeta(q,{j+p})} ((\gvf \circ_k \mu) \circ_j \psi) \bullet_0 t^{p-1}(x) 
\\
&
\qquad 
+ \sum^{p}_{j=1} (-1)^{\zeta(n,p) + \zeta(q,{j+p})} ((\mu \circ_1 \gvf) \circ_j \psi) \bullet_0 t^{p-1}(x) 
\\
&
\qquad
+ \sum^{p}_{j=1} (-1)^{\zeta({n-1},{p}) + \zeta(q,{j+p})} ((\mu \circ_2 \gvf) \circ_j \psi) \bullet_0 t^{p-1}(x) 
\\
&=
\sum^{p-1}_{j=1} \sum^{p-1}_{k=j} (-1)^{\xi({n},{k-1,p}) + \zeta(q,{j+p-1})} ((\gvf \circ_j \psi) \circ_k \mu) \bullet_0 t^{p-1}(x) 
\\
&\qquad
+  \sum^{p}_{j=1} (-1)^{\xi({n},{qj,p}) + \zeta(q,{p})} (\gvf \circ_j (\mu \circ_1 \psi)) \bullet_0 t^{p-1}(x) 
\\
&\qquad
+ \sum^{p-1}_{j=1} (-1)^{\zeta(n,{p}) + q(p-j)} (\gvf \circ_j (\mu \circ_2 \psi)) \bullet_0 t^{p-1}(x) 
+ (-1)^{\zeta(n,p)} (\gvf \circ_p (\mu \circ_2 \psi)) \bullet_0 t^{p-1}(x)
\\
&\qquad 
+ \sum^{p}_{j=3} \sum^{j-2}_{k=1} (-1)^{\xi({n},{k-1,p}) + \zeta(q,{j+p})} ((\gvf \circ_k \mu) \circ_j \psi) \bullet_0 t^{p-1}(x) 
\\
&\qquad
- \sum^{p}_{j=1} (-1)^{\xi({n},{qj,p}) + \zeta(q,{p})} (\gvf \circ_j (\mu \circ_1 \psi)) \bullet_0 t^{p-1}(x) 
\\
&\qquad 
+ \sum^{p}_{j=2} (-1)^{\zeta({n},{p}) + q(p-j+1)} (\gvf \circ_{j-1} (\mu \circ_2 \psi)) \bullet_0 t^{p-1}(x) 
\\
&\qquad
+ \sum^{p}_{j=1} (-1)^{\zeta({n},{p}) + \zeta(q,{j+p})} ((\mu \circ_1 \gvf) \circ_j \psi) \bullet_0 t^{p-1}(x) 
\end{split}
\end{equation*}
\begin{equation*}
\begin{split}
&\qquad 
+ \sum^{p}_{j=2} (-1)^{\zeta({n-1},{p}) + \zeta(q,{j+p})} ((\mu \circ_2 \gvf) \circ_j \psi) \bullet_0 t^{p-1}(x)
\\
&\qquad
+  (-1)^{\zeta({n-1},{p}) + \zeta({p-1},{q})} ((\mu \circ_2 \gvf) \circ_1 \psi) \bullet_0 t^{p-1}(x) \\
&=
\sum^{p-1}_{j=1} \sum^{p-1}_{k=1} (-1)^{\xi({n},{k-1,p}) + \zeta(q,{j+p-1})} ((\gvf \circ_j \psi) \circ_k \mu) \bullet_0 t^{p-1}(x) 
\\
&\qquad
+ (-1)^{\zeta(n,p)} (\gvf \circ_p (\mu \circ_2 \psi)) \bullet_0 t^{p-1}(x)
+ \sum^{p}_{j=1} (-1)^{\zeta({n},{p}) + \zeta(q,{j+p})} ((\mu \circ_1 \gvf) \circ_j \psi) \bullet_0 t^{p-1}(x) 
\\
&\qquad 
+ \sum^{p-1}_{j=1} (-1)^{\zeta({n-1},{p}) + \zeta(q,{j+p-1})} (\mu \circ_2 (\gvf \circ_j \psi)) \bullet_0 t^{p-1}(x)
\\
&\qquad
+  (-1)^{n(p-1)+p(q-1)} (\gvf \smallsmile \psi) \bullet_0 t^{p-1}(x) 
\end{split}
\end{equation*}
\end{footnotesize}
holds.
One notices now---taking the definition of the sign functions $\xi$ and $\zeta$ into account---that the third term in the equation above cancels with (6), such that the entire right-hand side of \rmref{lucca8} reads:
\begin{footnotesize}
\begin{equation}
\label{lucca12}
\begin{split}
& \sum^{p}_{j=1} (-1)^{\zeta(q,{j+p-1}) + np-1} (\mu \circ_2 (\gvf \circ_j \psi)) \bullet_0 t^p(x) 
\\
&\qquad
+ \sum^{p}_{j=1} \sum^p_{k=1}(-1)^{\zeta(q,{j+p-1}) + np + k} ((\gvf \circ_j \psi) \circ_k \mu) \bullet_0 t^p(x) \\
&= 
\sum^{p-1}_{j=1} (-1)^{\zeta(q,{j+p-1}) + n(p-1)} (\mu \circ_2 (\gvf \circ_j \psi)) \bullet_0 t^{p-1}(x) 
\\
&
\quad
+
\sum^{p-1}_{j=1} \sum^{p-1}_{k=1} (-1)^{\xi({n},{k-1,p}) + \zeta(q,{j+p-1})} ((\gvf \circ_j \psi) \circ_k \mu) \bullet_0 t^{p-1}(x) 
\\
&\quad 
+ (-1)^{\zeta(n,p)} (\gvf \circ_p (\mu \circ_2 \psi)) \bullet_0 t^{p-1}(x)
+  (-1)^{n(p-1) + p(q-1)} (\gvf \smallsmile \psi) \bullet_0 t^{p-1}(x). 
\end{split}
\end{equation}
\end{footnotesize}
At this point, the first two terms of the right-hand side of \rmref{lucca12} are similar to the right-hand side of \rmref{lucca8}, but with $p-1$ instead of $p$; whereas the third and the fourth term of \rmref{lucca12} are precisely the negative of the summands for $i=p$ of the {\em left}-hand side of \rmref{lucca8}. Hence, applying the same procedure as above another $p-2$ times to the first two terms of the right-hand side of \rmref{lucca12} produces the missing summands for $i=1, \ldots, p-1$ of the left-hand side of \rmref{lucca8}. This proves the theorem.
\end{proof}

\section{The Batalin-Vilkovisky module structure on a cyclic comp module}

The aim of this section is to obtain a sort of homotopy formula that holds for elements in the comp module $\cM$ with actions provided by arbitrary elements in the operad $\cO$ relating the Lie derivative \rmref{messagedenoelauxenfantsdefrance2} 
with the (cyclic) differential \rmref{extra2} and the cap product \rmref{alles4}. 
Since we want to hold this formula on the chain level of the mixed complex $(\cM_\bull, b, B)$, we need more ingredients first.

\subsection{The cyclic cap product}
\label{gallimard}

This subsection contains a  generalisation of the operator $S_\gvf$ that appeared in various, but narrower situations in the literature before \cite{Get:CHFATGMCICH, KowKra:BVSOEAT, NesTsy:OTCROAA, Rin:DFOGCA, Tsy:CH}.
As suggests a formal similarity of the homotopy formula \rmref{sacromonte1} in the next subsection with the identity $[b, \iota_\gvf] = \iota_{\gd\gvf}$ in \rmref{mulhouse2} above, we baptise this operator the {\em cyclic correction} to the cap product $\iota_\gvf$ and the sum $\iota_\gvf + S_\gvf$ the {\em cyclic cap product}. 

\begin{definition}
Let $(\cO, \mu)$ be an operad with multiplication and $\cM$ a \hbox{(para-)}cyclic unital comp module over $\cO$.
For every $\varphi \in  \cO(p)$, define 
$$
        S_\varphi : \cM(n) \rightarrow 
        \cM({n-p+2})
$$
for $0 \leq p \le n$ by
\begin{equation}
\label{capillareal1}
        S_\varphi := 
        \sum^{n-p+1}_{j=1} \, 
        \sum^{n - p+1}_{i=j} (-1)^{ n(j-1) + (p-1)(i-1)} e \bullet_0 \big(\gvf \bullet_i t^{j-1}(x)\big).
\end{equation}
If $p>n$, put
$
        S_\varphi := 0.
$
\end{definition}

\begin{rem}
For the identity $\mathbb{1} \in \cO(1)$, one obtains 
$S_{\mathbb{1}} = \sum^{n-1}_{j=0} (-1)^{jn} e \bullet_0 t^{j}(x)$ using \rmref{SchlesischeStr-1}, which differs from the cyclic coboundary $B$ as defined in \rmref{extra2} only by the last summand $(-1)^n e \bullet_0 t^n(x)$.
\end{rem}

Although looking sufficiently complicated, this operator has at least one nice property when passing to the normalised complex $\bar{\cM}_\bull$ ({\em cf.}~Remark \ref{normalised}):

\begin{lem}
Let $(\cO,\mu)$ be an operad with multiplication and $\cM$ a cyclic unital comp module over $\cO$. 
For any $\gvf \in \bar{\cO}^p$ in the normalised cochain complex 
({\em cf.}~Remark \ref{normalised}), the identity
\begin{equation}
\label{alles3}
[B, S_\varphi] = 0
\end{equation}
holds on the normalised chain complex $\bar{\cM}_\bull$.
\end{lem}
\begin{proof}
We need to show that $[B, S_\gvf] \subset {\rm im}(s_j)$ for $j = 0, \ldots, n$, where $s_j = e \bullet_{j+1} - $ are the degeneracies as given in \rmref{colleoppio}. First of all, for $p > n+1$, the entire expression is already zero. Assume therefore that $p \leq n+1$, and introduce the sign function
 \begin{equation}
\label{nerv1}
       \gvt({n,p},{j,i}) := 
        n(j-1) + (p-1)(i-1).
\end{equation} 
 By \rmref{SchlesischeStr} and \rmref{SchlesischeStr0} we have for $x \in \cM_n$, along with the expression \rmref{capillareal1} and \rmref{extra2} for $S_\gvf$ and $B$:
\begin{footnotesize}
\begin{equation*}
\begin{split}
S_\gvf B (x) &=   
        \sum^{n-p+2}_{j=1} \, 
        \sum^{n - p +2}_{i=j} \sum^{n}_{k=0} (-1)^{ \gvt({n,p},{j,i})  + kn} e \bullet_0 \big(\gvf \bullet_i t^{j-1}(e \bullet_0 t^k(x))\big) \\
&= \sum^{n-p+2}_{j=1} \, 
        \sum^{n- p +2}_{i=j} \sum^{n}_{k=0} (-1)^{ \gvt({n,p},{j,i})  + kn} e \bullet_0 \big(e \bullet_{j-1} (\gvf \bullet_{i-1} t^{k+j-1}(x))\big) \\
&= \sum^{n-p+2}_{j=1} \, 
        \sum^{n- p +2}_{i=j} \sum^{n}_{k=0} (-1)^{ \gvt({n,p},{j,i})  + kn} e \bullet_{j} \big(e \bullet_{0} (\gvf \bullet_{i-1} t^{k+j-1}(x))\big), 
\end{split}
\end{equation*}
\end{footnotesize}
and since $j$ is at least equal to $1$, the entire expression lies in the image of one of the $s_j$, for $j = 1, \ldots, n-p +2$. In the same fashion, it is even easier to show that this is true for the first term $B S_\gvf$ in the commutator $[B, S_\gvf]$ as well.   
Hence the lemma is proven.
\end{proof}

\subsection{The Cartan-Rinehart homotopy formula}
\label{lumograph}

With the ingredients of the previous subsection at hand, we can now state the main result of this article:

\begin{theorem}
\label{calleelvira}
Let $(\cO,\mu)$ be an operad with multiplication and $\cM$ a cyclic unital comp module over $\cO$. 
Then,
for any cochain $\varphi \in \bar \cO^\bull$ in the normalised cochain complex $\bar{\cO}^\bull$,
the homotopy formula
\begin{equation}
\label{sacromonte1}
\lie_\varphi = [\BB+\bb, \iota_\varphi + \SSS_\gvf] - \iota_{\gd\varphi} - \SSS_{\gd\varphi},
\end{equation}
or equivalently, in view of \rmref{alles3},
\begin{equation}
\label{sacromonte2}
\lie_\varphi = [\BB, \iota_\varphi] + [\bb, \SSS_\varphi] - \SSS_{\gd\varphi}
\end{equation}
holds on the normalised chain complex $\bar \cM_\bull$.
\end{theorem}

\begin{proof}
Also this proof consists basically in writing down all terms using the full battery of results obtained so far, plus some tedious double (triple, actually) sum yoga, and then comparing the sums one by one. First note that when explicitly expressing the graded commutators, Eq.~\rmref{sacromonte2} reads for $\gvf \in \bar{\cO}(p)$:
\begin{equation*}
\begin{split}
[\BB, \iota_\varphi] + [\bb, \SSS_\varphi] - \SSS_{\gd\varphi} 
&= \BB\iota_\varphi + (-1)^{p-1} \iota_\varphi \BB + \bb \SSS_\varphi + (-1)^{p-1} \SSS_\varphi \bb  - \SSS_{\gd\varphi}.
\end{split}
\end{equation*}
The statement in the cases $p > n+1$ and $p = n+1$ follows by definition. 
Recall the definition of the sign functions $\zeta$, $\xi$, and $\gvt$ in \rmref{nerv2} and \rmref{nerv1}, respectively.
For $p < n+1$ and $x \in \bar{\cM}(n)$, we then have:
\begin{footnotesize}
\begin{equation}
\label{calleelvira1}
\lie_\gvf(x) = \sum^{n-p+1}_{i=1} (-1)^{\zeta(p,i)} \gvf \bullet_i x + \sum^p_{i=1} (-1)^{\xi(n,{p,i})} \gvf \bullet_0 t^{i-1}(x) =: (1) + (2).
\end{equation}
\end{footnotesize}
On the other hand, using the properties \rmref{danton} of an operad and those stated in \rmref{distinguished2} of the operad multiplication, along with the properties 
\rmref{SchlesischeStr}--\rmref{lagrandebellezza2} 
of a cyclic unital comp module over an operad, as well as the definitions \rmref{alles4}, \rmref{extra2}, \rmref{appolloni2}, and \rmref{capillareal1} of the respective operators $\iota_\gvf, B, b, S_\gvf$, one computes
\begin{footnotesize}
\begin{equation*}
\begin{split}
\label{calleelvira1a}
&(-1)^{p-1} \iota_\gvf B (x) = \sum^{n}_{j=0} (-1)^{nj+p-1} (\mu \circ_2 \gvf) \bullet_0 (e \bullet_0 t^j(x))  
=
\sum^{n}_{j=0} (-1)^{nj+p-1} (\mu \circ_1 e) \bullet_0 (\gvf \bullet_0 t^j(x)) \\
&
=
\sum^{n}_{j=0} (-1)^{nj+p-1} \gvf \bullet_0 t^j(x) = \sum^{p}_{j=1} (-1)^{\xi(n,{p,j})} \gvf \bullet_0 t^{j-1}(x) + \sum^{n}_{j=p} (-1)^{nj+p-1} \gvf \bullet_0 t^j(x) 
\\
&
=: (3) + (4), 
\end{split}
\end{equation*}
\end{footnotesize}
and one immediately notes that $(2) = (3)$. Furthermore,
\begin{footnotesize}
\begin{equation*}
\begin{split}
\label{calleelvira2}
B \iota_\gvf (x) &= \!
\sum^{n-p}_{j=0} (-1)^{j(n-p)} e \bullet_0 t^j\big(\mu \bullet_0 (\gvf \bullet_1 x)\big)  
= \!\! \sum^{n-p+1}_{j=1} (-1)^{\gvt({n,p-1},{j,j})} e \bullet_0 \big(\mu \bullet_{j-1} (\gvf \bullet_j t^{j-1}(x))\big) 
\\
&
=: (5),  
\end{split}
\end{equation*}
\end{footnotesize}
whereas
\begin{footnotesize}
\begin{equation*}
\begin{split}
\label{calleelvira3}
b S_\gvf (x) &= 
 \sum^{n-p+1}_{k=1} \, \sum^{n-p+1}_{j=1} \, \sum^{n-p+1}_{i=j} 
(-1)^{\gvt({n,p},{j,i})+k} \mu \bullet_k \big(e \bullet_0 (\gvf \bullet_i t^{j-1}(x))\big)  \\
&\quad+
\sum^{n-p+1}_{j=1} \, \sum^{n-p+1}_{i=j} 
(-1)^{\gvt({n,p},{j,i})} \mu \bullet_0 \big(e \bullet_0 (\gvf \bullet_i t^{j-1}(x))\big) 
\\
& 
\quad 
-
\sum^{n-p+1}_{j=1} \, \sum^{n-p+1}_{i=j} 
(-1)^{nj+p(i-1)} \mu \bullet_0 t \big(e \bullet_0 (\gvf \bullet_i t^{j-1}(x))\big) 
\\
& = 
\sum^{n-p+1}_{j=1} \, \sum^{n-p+1}_{i=j} 
(-1)^{\gvt({n,p},{j,i})+j} e \bullet_0 \big(\mu \bullet_{j-1} (\gvf \bullet_i t^{j-1}(x))\big)  
\\
& 
\quad
+ 
\sum^{n-p+1}_{k=1 \atop {k \neq j}} \, \sum^{n-p+1}_{j=1} \, \sum^{n-p+1}_{i=j} 
(-1)^{\gvt({n,p},{j,i})+k} e \bullet_0 \big(\mu \bullet_{k-1} (\gvf \bullet_i t^{j-1}(x))\big)  
\\
& 
\quad
+
\sum^{n-p+1}_{j=2} \, \sum^{n-p+1}_{i=j} 
(-1)^{\gvt({n,p},{j,i})} \gvf \bullet_i t^{j-1}(x)  +
\sum^{n-p+1}_{i=1} 
(-1)^{\zeta(p,i)} \gvf \bullet_i x
\\
& 
\quad 
-
\sum^{n-p+1}_{j=1} \, \sum^{n-p+1}_{i=j} 
(-1)^{nj + (p-1)i} t \big(\gvf \bullet_i t^{j-1}(x)\big) =: (6) + (7) + (8) + (9) + (10),
\end{split}
\end{equation*}
\end{footnotesize}
using \rmref{distinguished2} in the second step, along with the unitality of the comp module. Here, we already observe that $(9) = (1)$.
One moreover has
\begin{footnotesize}
\begin{equation*}
\begin{split}
\label{calleelvira4}
(6) &= 
\sum^{n-p}_{j=1} \, \sum^{n-p+1}_{i=j+1} 
(-1)^{\gvt({n,p},{j,i})+j} e \bullet_0 \big(\mu \bullet_{j-1} (\gvf \bullet_i t^{j-1}(x))\big)  
\\
&\quad
-
\sum^{n-p+1}_{j=1}  
(-1)^{\gvt({n,p-1},{j,j})} e \bullet_0 \big(\mu \bullet_{j-1} (\gvf \bullet_j t^{j-1}(x))\big) 
=: (11) + (12),
\end{split}
\end{equation*}
\end{footnotesize}
where one observes that $(12) = - (5)$ and therefore cancels. Moreover,
\begin{footnotesize}
\begin{equation}
\begin{split}
\label{calleelvira5}
\!\!\!\!\!(10) &= 
\sum^{n-p}_{j=1} \, \sum^{n-p}_{i=j} 
(-1)^{nj+(p-1)i -1} \gvf \bullet_{i+1} t^j(x)  
+
\sum^{n-p+1}_{j=1}  
(-1)^{nj+\zeta(n,p) - 1} t \big(\gvf \bullet_{n - p+1} t^{j-1}(x)\big),
\end{split}
\end{equation}
\end{footnotesize}
where one now sees that the first sum in \rmref{calleelvira5} cancels with $(8)$ and the second one with $(4)$, which is proven by inserting $t^{n+1} = \id$ before the operation $\gvf \bullet_{n-p+1} -$.
We also have
\begin{footnotesize}
\begin{equation*}
\begin{split}
\label{calleelvira6}
(7) &=
\sum^{n-p+1}_{j=2} \, \sum^{n-p+1}_{i=j} \, \sum^{j-1}_{k=1} 
(-1)^{\gvt({n,p},{j,i})+k} e \bullet_0 \big(\mu \bullet_{k-1} (\gvf \bullet_i t^{j-1}(x))\big)  \\
& \quad 
+
\sum^{n-p}_{j=1} \, 
\sum^{n-p+1}_{i=j} \, \sum^{n-p+1}_{k=j+1} 
(-1)^{\gvt({n,p},{j,i})+k} e \bullet_0 \big(\mu \bullet_{k-1} (\gvf \bullet_i t^{j-1}(x))\big)  
=: (13) + (14). 
\end{split}
\end{equation*}
\end{footnotesize}
At this stage of the proof, all terms in \rmref{sacromonte2} cancelled except 
$(11)$, $(13)$ and $(14)$, but the terms $S_\gvf b$ and $S_{\gd \gvf}$ still need to be computed. We therefore continue by
\begin{footnotesize}
\begin{equation*}
\begin{split}
\label{calleelvira7}
(-1)^{p-1} S_\gvf b (x) &= \sum^{n-p}_{j=1} \, \sum^{n-p}_{i=j} 
(-1)^{\zeta(n,j) + \zeta(p,{i-1})} e \bullet_0 \big(\gvf \bullet_i t^{j-1}(\mu \bullet_0 x)\big)  
\\
&
\quad
-
\sum^{n-p}_{j=1} \, \sum^{n-p}_{i=j} 
(-1)^{\zeta(n,{j-1}) + \zeta(p,{i-1})} e \bullet_0 \big(\gvf \bullet_i t^{j-1}(\mu \bullet_0 t(x))\big)  
\\
& 
\quad 
+
\sum^{n-1}_{k=1} \, \sum^{n-p}_{j=1} \, \sum^{n-p}_{i=j} 
(-1)^{\zeta(n,j) + \zeta(p,{i-1}) + k} e \bullet_0 \big(\gvf \bullet_i t^{j-1}(\mu \bullet_k x)\big) 
=: (15) + (16) + (17),
\end{split}
\end{equation*}
\end{footnotesize}
and one sees that
\begin{footnotesize}
\begin{equation*}
\begin{split}
\label{calleelvira8}
(15) = \sum^{n-p}_{j=1} \, \sum^{n-p}_{i=j} 
(-1)^{\zeta(n,j) + \zeta(p,{i-1})} e \bullet_0 \big(\mu \bullet_{j-1} (\gvf \bullet_{i+1} t^{j-1}(x))\big) = - (11), 
\end{split}
\end{equation*}
\end{footnotesize}
whereas
\begin{footnotesize}
\begin{equation*}
\begin{split}
\label{calleelvira9}
(16) + (17) &= 
\sum^{n-p}_{j=1} \, \sum^{n-p}_{i=j} 
(-1)^{\zeta(n,{j-1}) + \zeta(p,{i-1}) +1} e \bullet_0 \big(\gvf \bullet_i (\mu \bullet_{j-1} t^j(x))\big)  
\\
& 
\quad 
+
\sum^{n-p}_{j=1} \, \sum^{n-p}_{i=j} \, \sum^{n-1}_{k=n-j+1} 
(-1)^{\zeta(n,j) + \zeta(p,{i-1}) + k} e \bullet_0 \big(\gvf \bullet_i t^{j-1}t^{n-j+1}(\mu \bullet_{k-n+j-1} t^j(x))\big) 
\\
& 
\quad 
+
\sum^{n-p}_{j=1} \, \sum^{n-p}_{i=j} \, \sum^{n-j}_{k=1} 
(-1)^{\zeta(n,j) + \zeta(p,{i-1}) + k} e \bullet_0 \big(\gvf \bullet_i (\mu \bullet_{k+j-1} t^{j-1}(x))\big)
\\
& = 
\sum^{n-p}_{j=1} \, \sum^{n-p}_{i=j} \, \sum^{j-2}_{k=0} 
(-1)^{\gvt({n,p},{j-1,i-1})+k} e \bullet_0 \big(\gvf \bullet_i (\mu \bullet_{k} t^j(x))\big) 
\\
&
\quad
+
\sum^{n-p}_{j=1} \, \sum^{n-p}_{i=j} \, \sum^{n-1}_{k=j} 
(-1)^{\gvt({n,p},{j,i-1})+k} e \bullet_0 \big(\gvf \bullet_i (\mu \bullet_{k} t^{j-1}(x))\big)
\\
& =: (18) + (19). 
\end{split}
\end{equation*}
\end{footnotesize}
By \rmref{SchlesischeStr} again, the term $(18)$ cancels with $(13)$.
Furthermore, compute with the help of \rmref{erfurt}:
\begin{footnotesize}
\begin{equation*}
\begin{split}
\label{calleelvira10}
- S_{\gd \gvf}(x) &=
- \sum^{n-p}_{j=1} \, \sum^{n-p}_{i=j} 
(-1)^{\gvt({n,p-1},{j,i})} e \bullet_0 \big((\mu \circ_1 \gvf) \bullet_i  t^{j-1}(x)\big)  
\\
&
\quad
+
\sum^{n-p}_{j=1} \, \sum^{n-p}_{i=j} 
(-1)^{\gvt({n,p-1},{j,i-1})} e \bullet_0 \big((\mu \circ_2 \gvf) \bullet_i  t^{j-1}(x)\big)  
\\
& 
\quad 
+
\sum^{p}_{k=1} \, \sum^{n-p}_{j=1} \, \sum^{n-p}_{i=j} 
(-1)^{\gvt({n,p},{j,i-1})+k-1} e \bullet_0 \big((\gvf \circ_k \mu) \bullet_i  t^{j-1}(x)\big)  
=: (20) + (21) + (22). 
\end{split}
\end{equation*}
\end{footnotesize}
To finish the proof of the theorem, we need to show now that this cancels with $(14)$ and $(19)$. In fact, we have
\begin{footnotesize}
\begin{equation*}
\begin{split}
\label{calleelvira11}
(14) &=
\sum^{n-p}_{k=1} \, \sum^{k}_{j=1} \, \sum^{n-p+1}_{i=j}
(-1)^{\gvt({n,p},{j,i})+k-1} e \bullet_0 \big(\mu \bullet_k ( \gvf \bullet_i  t^{j-1}(x))\big)
\\
&= 
\sum^{n-p}_{k=2} \, \sum^{k-1}_{j=1} \, \sum^{k-1}_{i=j}
(-1)^{\gvt({n,p},{j,i})+k-1} e \bullet_0 \big(\mu \bullet_k (\gvf \bullet_i  t^{j-1}(x))\big)
\\ 
&
\quad + 
\sum^{n-p}_{k=1} \, \sum^{k}_{j=1} \, \sum^{n-p+1}_{i=k+2}
(-1)^{\gvt({n,p},{j,i})+k-1} e \bullet_0 \big(\mu \bullet_k (\gvf \bullet_i  t^{j-1}(x))\big)
\\
&
\quad +  
\sum^{n-p}_{k=1} \, \sum^{k}_{j=1} (-1)^{\gvt({n,p-1},{j,k})} e \bullet_0 \big(\mu \bullet_k (\gvf \bullet_k  t^{j-1}(x))\big)
\\
&
\quad
-
\sum^{n-p}_{k=1} \, \sum^{k}_{j=1} 
(-1)^{\gvt({n,p-1},{j,k-1})} e \bullet_0 \big(\mu \bullet_k (\gvf \bullet_{k+1}  t^{j-1}(x))\big)
=: (23) + (24) + (25) + (26).
\end{split}
\end{equation*}
\end{footnotesize}
Here, $(25)$ cancels with $(20)$, and $(26)$ with $(21)$, 
whereas by the aforementioned multiple sum transformations and \rmref{SchlesischeStr} one obtains:
\begin{footnotesize}
\begin{equation}
\begin{split}
\label{calleelvira12}
(23) &=
\sum^{n-p}_{k=2} \, \sum^{k-1}_{i=1} \, \sum^{i}_{j=1}
(-1)^{\gvt({n,p}{j,i})+k-1} e \bullet_0 \big(\gvf \bullet_i (\mu \bullet_{k+p-1}  t^{j-1}(x))\big) \\
& = \sum^{n-p-1}_{i=1} \, \sum^{i}_{j=1} \, \sum^{n-1}_{k=i+p}
(-1)^{\gvt({n,p},{j,i-1})+k} e \bullet_0 \big(\gvf \bullet_i (\mu \bullet_{k}  t^{j-1}(x))\big), 
\end{split}
\end{equation}
\end{footnotesize}
along with
\begin{footnotesize}
\begin{equation}
\begin{split}
\label{calleelvira13}
(24) &=
\sum^{n-p}_{k=1} \, \sum^{k}_{j=1} \, \sum^{n-p}_{i=k+1}
(-1)^{\gvt({n,p},{j,i-1})+k-1} e \bullet_0 \big(\gvf \bullet_i (\mu \bullet_{k}  t^{j-1}(x))\big) \\
& = 
\sum^{n-p-1}_{j=1} \, \sum^{n-p-1}_{k=j} \, \sum^{n-p}_{i=k+1}
(-1)^{\gvt({n,p},{j,i-1})+k-1} e \bullet_0 \big(\gvf \bullet_i (\mu \bullet_{k}  t^{j-1}(x))\big).
\end{split}
\end{equation}
\end{footnotesize}
{\em Finally}, applying the middle line in \rmref{SchlesischeStr} from right to left to $(22)$ and adding this to $(19)$, one concludes that
\begin{footnotesize}
\begin{equation*}
\begin{split}
\label{calleelvira14}
(19) + (22) &=
\sum^{n-p-1}_{j=1} \, \sum^{n-p}_{i=j+1} \, \sum^{i-1}_{k=j}
(-1)^{\gvt({n,p},{j,i-1})+k} e \bullet_0 \big(\gvf \bullet_i (\mu \bullet_{k}  t^{j-1}(x))\big) \\
&
\quad
+ 
\sum^{n-p}_{j=1} \, \sum^{n-p}_{i=j} \, \sum^{n-1}_{k=i+p}
(-1)^{\gvt({n,p},{j,i-1})+k-1} e \bullet_0 \big(\gvf \bullet_i (\mu \bullet_{k}  t^{j-1}(x))\big), 
\end{split}
\end{equation*}
\end{footnotesize}
and whereas the first term cancels (after some triple sum transformation) 
with the remaining term in \rmref{calleelvira13}, the second one cancels with the remaining term in \rmref{calleelvira12}, which finalises the proof of the theorem.
\end{proof}

From the graded Jacobi identity for the graded commutator $[.,.]$, one immediately deduces the useful property:

\begin{corollary}
For any $\gvf \in \bar \cO^\bull$, the identity
\begin{equation}
\label{wajanochnich}
[\lie_\gvf, B] = 0
\end{equation}
holds on $\bar \cM_\bull$.
\end{corollary}

On the other hand, passing to homology allows for a more familiar formula:

\begin{corollary}
Let $(\cO,\mu)$ be an operad with multiplication and $\cM$ a cyclic unital comp module over $\cO$. For a cocycle $\gvf \in \bar{\cO}^\bull$ the induced maps
$$
\lie_\gvf: H_\bull(\cM) \to H_{\bull-p+1}(\cM), 
\qquad  \iota_\gvf: H_\bull(\cM) \to H_{\bull-p}(\cM) 
$$
on homology fulfil the simplified Cartan-Rinehart homotopy formula
\begin{equation}
\label{firestarter}
\lie_\gvf = [B, \iota_\gvf].
\end{equation}
\end{corollary}

\begin{rem}
\label{cordoli}
This formula is well-known in classical differential geometry between differential forms and vector fields, that is, the classical Cartan homotopy \cite{Car:LTDUGDLEDUEFP}, where $B$ is the de Rham differential, 
and also 
in the context of associative algebras, {\em i.e.}, in the classical cyclic homology theory of algebras, where $B$ is Connes' cyclic coboundary. 
In the context of commutative algebras, 
it appeared for the first time in \cite{Rin:DFOGCA}, in \cite{NesTsy:OTCROAA, Get:CHFATGMCICH} for the noncommutative situation, 
and in special cases such as for $1$-cocycles in \cite{Goo:CHDATFL, Con:NCDG, Xu:NCPA}, see, in particular, the example in \S\ref{castropretorio1} below. 
All these mentioned approaches have recently been generalised to bialgebroids (with some Hopf structure that yields a cyclic operation) in \cite{KowKra:BVSOEAT}, see the example in \S\ref{castropretorio3}.
\end{rem}

\section{Examples}
\label{examples}

\subsection{Noncommutative differential calculi for associative algebras}
\label{castropretorio1}
Let $A$ be an associative $k$-algebra and $V$ be a unital $A$-ring with multiplication denoted by $(v,v') \mapsto v \cdot_\vauu v'$ 
equipped with an $(A,A)$-bimodule map $\gamma:V \to A$ that fulfils $\gamma(v)v' = v \cdot_\vauu v' = v\gamma(v')$, which means that $\gamma$ is in particular a morphism of $A$-rings. 
Define 
$$
\cM := C_\bull(A, A) := A^{\otimes \bull+1} \qquad \mbox{and} 
\qquad  
\cO := C^\bull(A,V) := \Hom_k(A^{\otimes \bull},V).
$$ 
Also set
\begin{footnotesize}
\begin{equation}
\label{maifest}
 ({\gvf} \circ_i {\psi})(a_1, \ldots, a_{p+q-1}) :=  
{\gvf}\big(a_1, \ldots, a_{i-1}, \gamma(\psi(a_{i}, \ldots, a_{i+q-1})), a_{i+q}, \ldots, a_{p+q-1}\big),
\end{equation}
\end{footnotesize}
where $\gvf \in C^p(A,V), \ \psi \in C^q(A,V)$.
Then, it was shown in \cite{Ger:TCSOAAR} that $C^\bull(A,V)$ together with 
\begin{equation}
\label{maifest0}
\mu(a \otimes b) := a1_\vauu b, \qquad \mathbb{1} := \id_\ahha(\cdot)1_\vauu, \qquad e:=1_\vauu, \qquad \mbox{for} \ a, b \in A, 
\end{equation}
forms an operad with multiplication. Furthermore, it is an easy task to check that $C_\bull(A,A)$ becomes a comp module over $C^\bull(A,V)$ by means of the comp module maps
\begin{equation}
\label{maifest1}
\gvf \bullet_i (a_0, \ldots, a_n) := \big(a_0, \ldots, a_{i-1}, \gamma(\gvf(a_{i}, \ldots, a_{i+p-1})), a_{i+p}, \ldots, a_n\big), 
\end{equation}
for $\gvf \in C^p(A,V)$ and $i = 1, \ldots, n-p+1$,
where we denoted elementary tensors in the chain space $C_n(A,A)$ 
by $(a_0, \ldots, a_n)$, as common. 

Equally straightforward is to verify that
$C_\bull(A,A)$
can be made into a cyclic comp module over $C^\bull(A,V)$ via
\begin{equation}
\label{maifest2}
\gvf \bullet_0 (a_0, \ldots, a_n) := \big(\gamma(\gvf(a_{0}, \ldots, a_{p-1})), a_p, \ldots, a_{n} \big), 
\end{equation}
along with the classical cyclic operator (cf.~\cite{Con:NCDG})
\begin{equation}
\label{maifest3}
t(a_0, \ldots, a_n) := (a_n, a_0, \ldots, a_{n-1}).
\end{equation}

One immediately sees that computing the differential $b$ in \rmref{appolloni2} by means of the simplicial pieces \rmref{colleoppio} leads to the standard Hochschild boundary, $\gd$ from \rmref{erfurt} yields the standard Hochschild coboundary multiplied by $(-1)^{p+1}$, whereas the cup product \rmref{cupco} is the opposite of the standard one, and $B$ in \rmref{extra2} is the (Connes-Rinehart-Tsygan) cyclic coboundary.

When passing to the normalised chain and cochain complexes \hbox{$\bar C_n(A,A) := A \otimes \bar A^{\otimes n}$} 
resp.~$\bar C^p(A,V) = \Hom_k(\bar A^{\otimes p}, V)$, 
where $\bar A := A/k1_\ahha$, the resulting noncommutative differential calculus in the sense of the preceding sections on the pair $(H^\bull(A,V), H_\bull(A,A))$ of Hochschild cohomology and homology 
coincides---up to the fact that we allow for more general coefficients $V$ here---with the one introduced in \cite{Rin:DFOGCA, NesTsy:OTCROAA, Tsy:CH, Get:CHFATGMCICH} 
and further discussed in \cite{TamTsy:NCDCHBVAAFC, Tsy:FCFC}, see also the comment at the beginning of \S\ref{gallimard} as well as Remark \ref{cordoli}. The operators $\iota_\gvf$, $S_\gvf$, and $\lie_\gvf$ that result from \rmref{alles4}, \rmref{capillareal1}, and \rmref{messagedenoelauxenfantsdefrance2}, respectively, are those from  \cite{Rin:DFOGCA, NesTsy:OTCROAA, Tsy:CH, Get:CHFATGMCICH} (see also \cite[\S7.2]{KowKra:BVSOEAT} for yet another generalisation): 
abbreviating $x := (a_0, \ldots, a_n)$, one obtains
\begin{footnotesize}
\begin{eqnarray}
\label{dunkeljetze1}
 \iota_{\varphi}(x) 
&\!\!\!\!\!=&\!\!\!\!\! \big(a_0 \gamma({\varphi}(a_1, \ldots, a_p)), a_{p+1}, \ldots,  a_n\big), \\
\nonumber
\SSS_{\varphi}(x) 
&\!\!\!\!\!=&\!\!\!\!\! \sum^{n-p+1}_{j=1} \sum^{n-p+1}_{i=j} (-1)^{n(j-1)+(p-1)(i-1)} 
\big(1_\ahha, a_{n-j+2}, \ldots, a_n, a_0, \ldots, a_{i-1}, \\
\label{dunkeljetze2}       
&& 
\qquad\qquad\qquad
\gamma(\gvf(a_{i}, \ldots, a_{i+p-1})), a_{i+p}, \ldots, a_{n-j+1}\big)
\\
\label{dunkeljetze3}
\lie_{\varphi}(x) 
&\!\!\!\!\!=&\!\!\!\!\! 
\nonumber
\sum^{n-p+1}_{i=1} (-1)^{(p-1)(i-1)}  
\big(a_0, \ldots, a_{i-1}, \gamma(\varphi(a_{i}, \ldots, a_{i+p-1})), a_{i+p}, \ldots, a_n\big) \\
&& \!\!\!\!\!\!\!\!\!\! + \!\!
\sum^{p}_{i=1} (-1)^{n(i-1)+(p-1)}  
\big(\gamma(\varphi(a_{n-i+2}, \ldots, a_n, a_0, \ldots, a_{i+p-2})), a_{i+p-1}, \ldots, a_{n-i+1}\big). 
\end{eqnarray}
\end{footnotesize}
In \S\ref{castropretorio3} in view of Remark \ref{genausowillich} we discuss the possibility to introduce coefficients $W$ even in the chain complex $\bar C_\bull(A,W)$ so as still to obtain a calculus structure on (co)homology.

\subsection{Noncommutative Poisson structures}
\label{castropretorio2}
In a situation similar to the preceding subsection, {\em i.e.}, for the Hochschild chains $\cM := C_\bull(A,A)$ and cochains $\cO := C^\bull(A,V)$, where $V$ is a unital $A$-ring, equipped with the comp maps \rmref{maifest} and the (cyclic) comp module maps \rmref{maifest1}--\rmref{maifest3}, and given a morphism of $A$-rings $\gamma: V \to A$, replace the operad multiplication $\mu$ given in \rmref{maifest0} 
by the $(A,A)$-bimodule structure of $V$ 
by a more general structure $\pi \in C^2(A,V)$ that fulfils 
$$
\pi \circ_1 \pi = \pi \circ_2 \pi,
$$
as in Eq.~\rmref{distinguished1}. 
In case $\pi$ is a Hochschild co{\em cycle} (and $V:=A$ as well as the characteristic of $k$ different from two), such a structure was called a {\em noncommutative Poisson structure} in \cite{Xu:NCPA},
and if it moreover fulfils 
$\pi(1_\ahha,1_\ahha) = 1_\vauu$, it is straightforward to see that \rmref{distinguished2} holds as well 
if we set, as above, $\mathbb{1} = \id_\ahha(\cdot)1_\vauu$ and $e = 1_\vauu$. Hence, $(C^\bull(A,V), \pi)$ becomes an operad with multiplication.

With these assumptions, the simplicial differential in \rmref{appolloni2} then reads
\begin{footnotesize}
\begin{equation*}
\begin{split}
b^\pi  (a_0, \ldots, a_n) &=
\sum^{n-1}_{i=0} (-1)^{i}  
\big(a_0, \ldots, a_{i-1}, \gamma(\pi(a_{i}, a_{i+1})), a_{i+2}, \ldots, a_n\big) 
\\
&\qquad
+ (-1)^{n}  
\big(\gamma(\pi(a_{n}, a_0)), a_1, \ldots, a_{n-1}\big),
\end{split}
\end{equation*}
\end{footnotesize}
which is---again up to the fact that we allow for more general coefficients---the algebraic version, introduced in \cite{NeuPflPosTan:HOFDOPELG}, of the geometric {\em Brylinski boundary} that defines Poisson homology. Its original definition in \cite{Bry:ADCFPM} 
as 
$$
b^\pi := [\iota_\pi, d_{\scriptscriptstyle{\rm deRham}}], 
$$
where $\iota$ is defined as in \rmref{dunkeljetze1}, is nothing than the homotopy formula \rmref{firestarter} combined with \rmref{sachengibts}.
The cosimplicial differential that results from \rmref{erfurt}, on the other hand, is given as
\begin{footnotesize}
\begin{equation*}
\begin{split}
(\gb^\pi \gvf)(a_1, \ldots, a_{p+1}) &=
\pi\big(\gamma(\gvf(a_1, \ldots, a_p)), a_{p+1}\big) 
+ (-1)^{p-1}\pi\big(a_1, \gamma(\gvf(a_2, \ldots, a_{p+1}))\big)
\\ 
&\qquad + \sum^{p}_{i=1} (-1)^{i+p-1}  
\gvf\big(a_1, \ldots, a_{i-1}, \gamma(\pi(a_{i}, \ldots, a_{i+1})), a_{i+2}, \ldots, a_{p+1}\big),  
\end{split}
\end{equation*}
\end{footnotesize}
which is the algebraic version, introduced in \cite{Xu:NCPA}, of the geometric Koszul-Lichnerowicz coboundary \cite{Kos:CDSNEC, Lic:LVDPELADLA} that defines Poisson cohomology.

From \rmref{sachengibts} and \rmref{wajanochnich} follows now that $(C_\bull(A,A), b^\pi, B)$ defines a mixed complex and one could study its cyclic homology as in \cite{Pap:HAAVDP}, or its connection to Koszul operads as in \cite{Fre:TDODKEHDADP}, see also the considerations in \cite{GetKap:COACH}; we will be satisfied by stating 
the remaining maps of the calculus structure introduced in the preceding sections: since we defined the comp module maps and the cyclic operator in \rmref{maifest1}--\rmref{maifest3} as in the example on Hochschild (co)chains, the operators $S_\gvf$ and $\lie_\gvf$ look as in \rmref{dunkeljetze2} resp.~\rmref{dunkeljetze3}. On the other hand, the cup and cap product \rmref{cupco} resp.~\rmref{alles4} assume a different form: 
\begin{footnotesize}
\begin{eqnarray*}
(\gvf \smallsmile \psi)(a_1, \ldots, a_{p+q})  
&\!\!\!\!\!=&\!\!\!\!\! 
\pi\big(\gamma(\psi(a_1, \ldots, a_q)), \gamma(\varphi(a_{q+1}, \ldots, a_{p+q}))\big), \\
\gvf \smallfrown (a_0, \ldots, a_n) &\!\!\!\!\! = &\!\!\!\!\!   
\big(\gamma(\pi(a_0, \gamma(\varphi(a_1, \ldots, a_p)))), a_{p+1}, \ldots, a_{n}\big), 
\end{eqnarray*}
\end{footnotesize}
for $\gvf \in C^p(A,V)$, $\psi \in C^q(A,V)$,
which one may call the {\em Poisson cup product} and {\em Poisson cap product}, respectively.

\subsection{Noncommutative differential calculi for left Hopf algebroids}
\label{castropretorio3}

The main example for cyclic comp modules over an operad 
is the case of a left Hopf algebroid, which even comprises the preceding two examples. 

In \cite{KowKra:BVSOEAT} the Gerstenhaber and Batalin-Vilkovisky structures arising from bialgebroids resp.\ left Hopf algebroids were discussed in detail, and in the following we indicate how these can be recovered from what was proven in the 
previous sections. Moreover, using the results of \cite{Kow:BVASOCAPB}, 
we generalise the noncommutative differential calculus structure by introducing certain coefficients in the operad structure.

\subsubsection{Bialgebroids and left Hopf algebroids}

Since the precise definitions of bialgebroids, left Hopf algebroids, their modules, and comodules are quite technical and would correspondingly fill pages, we will be rather cursory in this subsection and refer for any additional information to the literature cited below, or directly to \cite{KowKra:BVSOEAT, Kow:BVASOCAPB}, where all information and notation needed in what follows is explained in sufficient detail.

In this section, both $U$ and $A$ are unital and associative $k$-algebras, together with a $k$-algebra map $\eta: \Ae \to U$, where $\Ae = A \otimes \Aop$. 
This map induces forgetful functors from the categories $\umod$ and $\modu$ of left resp.\ right $U$-modules to the category $\amoda$ of left $\Ae$-modules. More precisely, introduce the notation
$$
        a \lact n \ract b:=\eta(a \otimes b)n,\quad
        a \blact m \bract b:=m \eta(b \otimes _k a),\quad
        a,b \in A,n \in N,m \in M,
$$
for $N \in \umod$ and $M \in \modu$.
We will frequently use this notation to indicate the relevant $A$-module structures when defining, for example, tensor products. In particular, these forgetful functors can be applied to $U$ itself with respect to left and right multiplication, and this defines two morphisms
\begin{equation*}
\label{basmati}
        s,t : A \rightarrow U,\quad s(a):=\eta (a \otimes 1),\quad
        t(b):=\eta (1 \otimes b), \qquad a, b \in A,
\end{equation*}
the {\em source} resp.\ {\em target} map of the pair $(U,A)$.

Recall from \cite{Schau:DADOQGHA} that a {left Hopf algebroid} generalises the notion of Hopf algebra and consists of a bialgebroid plus a certain (Hopf-)Galois map which is required to be invertible, see below. A {\em bialgebroid} \cite{Tak:GOAOAA}, in turn, is a sextuple $(U,A,s,t,\Delta, \gve)$, abbreviated $(U,A)$, that generalises a bialgebra (the case of which is recovered by taking $A=k$), and which adds to the data $(U,A, s, t)$ introduced above two maps, the {\em coproduct} and {\em counit} 
\begin{equation}
\label{tuscania1}
\Delta: U \to \due U {} \ract  \otimes_\ahha \due U \lact {},  
\qquad \gve: U \to A,
\end{equation}
subject to a couple of identities technically slightly more intricate than the corresponding ones for bialgebras, and for the details of which we refer to, for example, \cite[Def.~3.3]{Boe:HA}.
As common, we introduce a Sweedler presentation for the image of the coproduct, that is, $\gD(u) =: u_{(1)} \otimes_\ahha u_{(2)}$, with summation understood. 

A {\em left Hopf algebroid} now adds to the structure of a bialgebroid the requirement that the Galois map 
$\gb: \due U \blact {} \otimes_\Aopp U_\ract \to \due U {} \ract \otimes_\ahha \due U \lact {}, \ u \otimes_\Aopp v \to u_{(1)} \otimes_\ahha u_{(2)} v$ be invertible, and this allows to define a so called 
{\em translation map}
\begin{equation}
\label{tuscania2}
U \to \due U \blact {} \otimes_\Aopp U_\ract, \quad u \mapsto \beta^{-1}(u \otimes_A 1) := u_+ \otimes_\Aopp u_-,
\end{equation}
for which we again introduced a Sweedler notation as above. An interesting feature in case this extra structure is present is \cite[Lem.~3]{KowKra:DAPIACT} that now the tensor product 
$\due M \blact {} \otimes_\Aopp N_\ract$ 
of $N \in \umod$ with $M \in \modu$ allows for a sort of adjoint action, {\em i.e.}, becomes a right $U$-module again, by means of the action
\begin{equation}
\label{corelli1}
(m \otimes_\Aopp n)u := mu_+ \otimes_\Aopp u_-n, \qquad m \in M, n \in N, u \in U.
\end{equation}
Observe that the induced right $\Ae$-module structure is, by the properties of the translation map \cite[Prop.~3.7]{Schau:DADOQGHA}, forced to be
\begin{equation}
\label{corelli2}
a \blact (m \otimes_\Aopp n) \bract b = m \bract b \otimes_\Aopp a \lact n, \qquad m \in M, \ n \in N, \ a, b \in A. 
\end{equation}


\subsubsection{(Anti) Yetter-Drinfel'd modules and Yetter-Drinfel'd algebras}

Remember \cite{Schau:DADOQGHA} that a {\em Yetter-Drinfel'd module} $N$ for a left bialgebroid $(U,A)$ is both a left $U$-module with action $(u,n) \mapsto un$ and a left $U$-comodule with 
coaction $\gD_\enne: N \to U_\ract \otimes_\ahha N, \ n \mapsto n_{(-1)} \otimes_\ahha n_{(0)}$, 
that obey a certain compatibility condition both for the underlying $\Ae$-module structure arising on $N$ 
from action and coaction by forgetful functors,  
and a compatibility condition between $U$-action 
and $U$-coaction (see, again, \cite{Boe:HA} for technical details, or in the proof of Lemma \ref{altemps} below). 
One can show \cite{Schau:DADOQGHA} that Yetter-Drinfel'd modules form a monoidal category $\yd$ and that this category is equivalent to the weak centre of the category $\umod$, which implies the existence of a braiding $\gs$. A {\em Yetter-Drinfel'd algebra} (see, {\em e.g.}, \cite{BrzMil:BBAD}) is a monoid in $\yd$, and it is called {\em braided commutative} if it is commutative with respect to the braiding $\gs$.

Finally, an {\em anti Yetter-Drinfel'd module} for a left Hopf algebroid $(U,A)$ 
is both a right $U$-module $M$ with action $(m,u) \mapsto mu$ and a left $U$-comodule with coaction $\Delta_\emme(m) := m_{(-1)} \otimes_\ahha m_{(0)}$ such that the underlying 
$\Ae$-module structures of action and coaction coincide and a certain compatibility between action and coaction is fulfilled (see \cite{BoeSte:CCOBAAVC} and Eq.~\rmref{huhomezone2} below). 
Such an anti Yetter-Drinfel'd module is called {\em stable} if $m = m_{(0)} m_{(-1)}$ holds. Observe that the category $\ayd$ of anti Yetter-Drinfel'd modules usually is {\em not} monoidal.

A fact that will be used below is the following:

\begin{lem}
\label{altemps}
For $M \in \ayd$ and $N \in \yd$, the tensor product $M \otimes_\Aopp N$ equipped with the canonical left $U$-coaction 
and the right $U$-action \rmref{corelli1} forms an anti Yetter-Drinfel'd module again.
\end{lem}
\begin{proof}
The left $U$-coaction in question is given by 
\begin{eqnarray*}
\gD_{\emme \otimes \enne}: \due M \blact {} \otimes_\Aopp N_\ract
&\to& U_\ract \otimes_\ahha \due {(\due M \blact {} \otimes_\Aopp N_\ract)} {\blact} {} \! , 
\\
 m \otimes_\Aopp n &\mapsto& n_{(-1)}m_{(-1)} \otimes_\ahha m_{(0)} \otimes_\Aopp n_{(0)}, 
\end{eqnarray*}
where the $\Ae$-module structure \rmref{corelli2} is used on $\due M \blact {} \otimes_\Aopp N_\ract$ (observe the opposite factor ordering; that this map is well-defined depends on the fact that bialgebroid coactions corestrict to certain {\em Sweedler-Takeuchi subspaces}, see, {\em e.g.}, \cite[\S2 \& \S3.1]{Kow:BVASOCAPB} for details). In case $U$ is a left Hopf algebroid, the Yetter-Drinfel'd property of $N$ means
\begin{equation}
\label{huhomezone1b}
\gD_\enne(un)
= u_{+(1)} n_{(-1)} u_{-} \otimes_\ahha u_{+(2)} n_{(0)},
\end{equation}
whereas the anti Yetter-Drinfel'd property for $M$ reads 
\begin{equation}
\label{huhomezone2}
		\gD_\emme(mu) 
		  =  u_- m_{(-1)} u_{+(1)} \otimes_\ahha m_{(0)} u_{+(2)},
\end{equation} 
which is what we have to show to hold for $M \otimes_\Aopp N$ as well. Using \rmref{corelli1}, we obtain:
\begin{equation*}
\begin{split}
\gD_{\emme \otimes \enne}&\big((m \otimes_\Aopp n)u\big) 
\\
&
= \gD_{\emme \otimes \enne}(mu_+ \otimes_\Aopp u_-n) 
\\
&= 
(u_-n)_{(-1)}(mu_+)_{(-1)} \otimes_\ahha (mu_+)_{(0)} \otimes_\Aopp (u_-n)_{(0)} 
\\
&= 
u_{-+(1)} n_{(-1)} u_{--} u_{+-}  m_{(-1)} u_{++(1)} \otimes_\ahha m_{(0)} u_{++(2)} \otimes_\Aopp u_{-+(2)} n_{(0)} 
\\
&= 
u_{-(1)} n_{(-1)} m_{(-1)} u_{+(1)} \otimes_\ahha m_{(0)} u_{+(2)} \otimes_\Aopp u_{-(2)} n_{(0)} 
\\
&= 
u_{-} n_{(-1)} m_{(-1)} u_{+(1)} \otimes_\ahha m_{(0)} u_{+(2)+} \otimes_\Aopp u_{+(2)-} n_{(0)} 
\\
&= 
u_{-} n_{(-1)} m_{(-1)} u_{+(1)} \otimes_\ahha (m_{(0)} \otimes_\Aopp n_{(0)})  u_{+(2)},
\end{split}
\end{equation*}
where we used the table of identities given in \cite[Prop.~3.7]{Schau:DADOQGHA} or \cite[Eqs.~(2.4)--(2.12)]{KowKra:BVSOEAT} 
for the translation map \rmref{tuscania2}. 
\end{proof}

\subsubsection{The operad arising from a bialgebroid}

Let $(U,A)$ be a left bialgebroid and let $N$ be a braided commutative Yetter-Drinfel'd algebra with ring structure denoted by $(n, n') \mapsto n \cdot_\enne n'$. 
Set
$$
C^\bull(U,N) := \Hom_\Aopp\big(({\due U \blact \ract}\!)^{\otimes_\Aopp \bull}, N\big).
$$
Then it was shown in \cite[Thm.\ 3.1]{Kow:BVASOCAPB} that $C^\bull(U,N)$ defines an operad with multiplication with respect to the following structure maps (observe that we use a different ordering convention here, which, however, does not change the statement):
define the system of comp maps
$$
        \circ_i : C^p(U,N) \otimes C^q(U,N) \rightarrow 
        C^{p+q-1}(U,N),\quad
        i = 1, \ldots, p,
$$
by
\begin{footnotesize}
\begin{equation}
\label{maxdudler1}
\begin{split}
        & (\varphi \circ_i \psi)(u^1, \ldots, u^{p+q-1}) 
\\ &
\! := \varphi(u^1_{(1)}, \ldots, u^{p-i}_{(1)}, 
[\psi(u^{p-i+1}_{(1)}, \ldots, u^{p+q-i}_{(1)})]_{(-1)} u^{p-i+1}_{(2)} \cdots u^{p+q-i}_{(2)},        
u^{p+q-i+1}, \ldots, u^{p+q-1}) \\
&\hspace*{4cm} \cdot_\enne \big(u^1_{(2)} \cdots u^{p-i}_{(2)} 
[\psi(u^{p-i+1}_{(1)}, \ldots, u^{p+q-i}_{(1)})]_{(0)}
\big), 
\end{split}
\end{equation}
\end{footnotesize}
where we denoted elementary tensors in ${(\due U \blact \ract\!)}^{\otimes_\Aopp k}$ by $(u^1, \ldots, u^k)$.
For zero cochains $n \in N$, this has to be read as $n \circ_i \psi = 0$ for all $i$ and
all $\psi$, whereas 
\begin{footnotesize}
\begin{equation}
\label{maxdudler2}
(\gvf \circ_i n)(u^1, \ldots, u^{p-1}) :=  \varphi(u^1_{(1)}, \ldots, u^{p-i}_{(1)}, 
        n_{(-1)},u^{p-i+1}, \ldots, u^{p-1}) \cdot_\enne (u^1_{(2)} \cdots u^{p-i}_{(2)} n_{(0)})
\end{equation}
\end{footnotesize}
for $i = 1, \ldots, p$, producing an element in $C^{p-1}(U,N)$.
Together with
\begin{equation}
\label{maxdudler3}
        \mu := \varepsilon(m_\uhhu(\cdot,\cdot)) \lact 1_\enne, \quad
\mathbb{1} := \gve(\cdot) \lact 1_\enne, \quad \mbox{and} \quad  
e := 1_\enne,
\end{equation} 
where $m_\uhhu$ is the multiplication map of $U$, Eqs.~\rmref{maxdudler1}--\rmref{maxdudler3}, as mentioned above, 
define an operad with multiplication in the sense of Definition \ref{moleskine}.

\subsubsection{The operad module arising from a bialgebroid}

Now we assume the situation of Lemma \ref{altemps}: let $(U,A)$ be not only a left bialgebroid but rather a left Hopf algebroid and 
$M \in \ayd$. Together with the braided commutative Yetter-Drinfel'd algebra $N \in \yd$ from the preceding subsection, we can form the anti Yetter-Drinfel'd module $\due M \blact {} \otimes_\Aopp N_\ract$. Set
$$
C_\bull(U, M \otimes_\Aopp N) := \due {(\due M \blact {} \otimes_\Aopp N_\ract)} \blact {} \otimes_\Aopp {\due U \blact \ract}^{\otimes_\Aopp \bull},
$$
and if we denote elementary tensors in $C_k(U,M \otimes_\Aopp N)$, analogously to above, by $(m, n, u^1, \ldots, u^k)$, we can define
the comp module maps
\begin{footnotesize}
\begin{equation}
\begin{split}
\label{subiaco1}
\gvf & \bullet_i (m, n, u^1, \ldots, u^k) \\
& :=  \big(m, (u^{1}_{(2)} \cdots u^{k-p-i+1}_{(2)} [\gvf(u^{k-p-i+2}_{(1)}, \ldots, u^{k-i+1}_{(1)})]_{(0)})
 \cdot_\enne n, 
 u^1_{(1)}, \ldots, u^{k-p-i+1}_{(1)}, 
\\
&
\qqquad 
[\gvf(u^{k-p-i+2}_{(1)}, \ldots, u^{k-i+1}_{(1)})]_{(-1)} u^{k-p-i+2}_{(2)} \cdots u^{k-i+1}_{(2)}, u^{k-i+2}, \ldots, u^k \big),
\end{split}
\end{equation}
\end{footnotesize}
for $i=1, \ldots, k-p+1$, 
and any $\gvf \in C^p(U,N)$, where $p > 0$. For $p=0$, that is, elements in $N$, define
\begin{footnotesize}
\begin{equation}
\begin{split}
\label{subiaco2}
n' & \bullet_i (m, n, u^1, \ldots, u^k) 
 \\
 &
 \quad
:=  \big(m, (u^{1}_{(2)} \cdots u^{k-i+1}_{(2)} n'_{(0)})
 \cdot_\enne n, u^1_{(1)}, \ldots, u^{k-i+1}_{(1)}, 
n'_{(-1)}, u^{k-i+2}, \ldots, u^{k} \big),
\end{split}
\end{equation}
\end{footnotesize}
for $i=1, \ldots, k+1$ and $n' \in N$. That \rmref{subiaco1}--\rmref{subiaco2} and also \rmref{maxdudler1}--\rmref{maxdudler2} are well-defined depends on the fact that $N$ is in particular a monoid in $\umod$ (see, in particular, \cite[Eq.~(2.28)]{Kow:BVASOCAPB}), Eq.~\rmref{corelli2}, and the $\Ae$-linearity of the coproduct and the coaction, along with the fact that both coproduct and coaction corestrict, as mentioned before, to a certain 
{\em Sweedler-Takeuchi subspace}, see again \cite[\S2 \& \S3.1]{Kow:BVASOCAPB} for a more detailed explanation.

We are now in a position to state the main result of this subsection:

\begin{theorem}
\label{zanjrevolution}
Let $(U,A)$ be a left Hopf algebroid, $M$ an anti Yetter-Drinfel'd module and $N$ a braided commutative Yetter-Drinfel'd algebra. Then $C_\bull(U,M \otimes_\Aopp N)$ forms a para-cyclic module over the operad $C^\bull(U,N)$ with multiplication with respect to the operations \rmref{subiaco1}--\rmref{subiaco2} 
and the extra comp module map
\begin{footnotesize}
\begin{equation}
\begin{split}
\label{subiaco3}
\gvf & \bullet_0 (m, n, u^1, \ldots, u^k) \\
& :=  \big(m_{(0)}, \big(u^{1}_{+(2)} \cdots u^{k-p+1}_{+(2)} \gvf(u^{k-p+2}_+, \ldots, u^{k}_+, u^k_- \cdots u^1_- n_{(-1)} m_{(-1)})\big)
 \cdot_\enne n_{(0)}, 
\\
&
\qqquad
 u^1_{+(1)}, \ldots, u^{k-p+1}_{+(1)}\big), 
\end{split}
\end{equation}
\end{footnotesize}
where $\gvf \in C^p(U,N)$, along with the cyclic operator
\begin{footnotesize}
\begin{equation}
\label{subiaco4}
t(m,n,u^1, \ldots, u^k) := (m_{(0)} u^1_{++}, u^1_{+-}n_{(0)}, u^2_+, \ldots, u^k_+, u^k_- \cdots u^1_- n_{(-1)} m_{(-1)}).
\end{equation}
\end{footnotesize}
If 
the anti Yetter-Drinfel'd module $M \otimes_\Aopp N$ is stable, then $C_\bull(U,M \otimes_\Aopp N)$ is a cyclic unital module over the operad $C^\bull(U,N)$.
\end{theorem}

\begin{proof}
We need to verify that \rmref{SchlesischeStr}--\rmref{SchlesischeStr-1} are true for the operations \rmref{subiaco1}--\rmref{subiaco3}, along with \rmref{lagrandebellezza1}--\rmref{lagrandebellezza2} with respect to the cyclic operator \rmref{subiaco4}. 
As for \rmref{SchlesischeStr}, these identities have to be checked separately for $i = 0$ and $i > 0$ since the respective operations $\bullet_i$ are of different form. 
However, for obvious reasons of space we only show the (more intricate) case for $i=0$ and $0 \leq j < p$, as all other cases can be seen by similar (and similarly tedious, although sometimes easier) computations.
One has for $\gvf \in C^p(U,M)$ and $\psi \in C^q(U,M)$:
\begin{footnotesize}
\begin{equation*}
\begin{split}
& \gvf \bullet_0 (\psi \bullet_j (m,n,u^1, \ldots, u^k)) 
\\
&
= \gvf \bullet_0 \big(m, (u^{1}_{(2)} \cdots u^{k-q-j+1}_{(2)} [\psi(u^{k-q-j+2}_{(1)}, \ldots, u^{k-j+1}_{(1)})]_{(0)})
 \cdot_\enne n, 
 u^1_{(1)}, \ldots, u^{k-q-j+1}_{(1)}, 
\\
&
\qqquad \qqquad \qqquad
[\psi(u^{k-q-j+2}_{(1)}, \ldots, u^{k-j+1}_{(1)})]_{(-1)} u^{k-q-j+2}_{(2)} \cdots u^{k-j+1}_{(2)}, u^{k-j+2}, \ldots, u^k \big)
\\
&
= 
\big(m_{(0)}, \Big(u^{1}_{(1)+(2)} \cdots u^{k-q-p+2}_{(1)+(2)} \gvf\Big(u^{k-q-p+3}_{(1)+}, \ldots, u^{k-q-j+1}_{(1)+}, 
\\
&
\qquad
[\psi(u^{k-q-j+2}_{(1)}, \ldots, u^{k-j+1}_{(1)})]_{(-1)+} u^{k-q-j+2}_{(2)+} \cdots u^{k-j+1}_{(2)+}, u^{k-j+2}_+, \ldots, u^k_+,
\\
&
\qquad
u^k_- \cdots u^{k-j+2}_- u^{k-j+1}_{(2)-} \cdots u^{k-q-j+2}_{(2)-}[\psi(u^{k-q-j+2}_{(1)}, \ldots, u^{k-j+1}_{(1)})]_{(-1)-} u^{k-q-j+1}_{(1)-} \cdots u^1_{(1)-}
\\
&
\qquad
\big[(u^{1}_{(2)} \cdots u^{k-q-j+1}_{(2)} [\psi(u^{k-q-j+2}_{(1)}, \ldots, u^{k-j+1}_{(1)})]_{(0)})
 \cdot_\enne n\big]_{(-1)}
m_{(-1)}\Big)\Big)
\\
&
\qquad
\cdot_\enne \big[(u^{1}_{(2)} \cdots u^{k-q-j+1}_{(2)} [\psi(u^{k-q-j+2}_{(1)}, \ldots, u^{k-j+1}_{(1)})]_{(0)})
 \cdot_\enne n\big]_{(0)}, 
u^{1}_{(1)+(1)}, \ldots, u^{k-q-p+2}_{(1)+(1)}\big)
\\
&
=
\Big(m_{(0)}, \Big(u^{1}_{(1)+(2)} \cdots u^{k-q-p+2}_{(1)+(2)} \gvf\Big(u^{k-q-p+3}_{(1)+}, \ldots, u^{k-q-j+1}_{(1)+}, 
\\
&
\qquad
[\psi(u^{k-q-j+2}_{(1)}, \ldots, u^{k-j+1}_{(1)})]_{(-1)+} u^{k-q-j+2}_{(2)+} \cdots u^{k-j+1}_{(2)+}, u^{k-j+2}_+, \ldots, u^k_+,
\\
&
\qquad
u^k_- \cdots u^{k-j+2}_- u^{k-j+1}_{(2)-} \cdots u^{k-q-j+2}_{(2)-}[\psi(u^{k-q-j+2}_{(1)}, \ldots, u^{k-j+1}_{(1)})]_{(-1)-} u^{k-q-j+1}_{(1)-} \cdots u^1_{(1)-}
\\
&
\qquad
\big[u^{1}_{(2)} \cdots u^{k-q-j+1}_{(2)} [\psi(u^{k-q-j+2}_{(1)}, \ldots, u^{k-j+1}_{(1)})]_{(0)}\big]_{(-1)}
n_{(-1)}
m_{(-1)}\Big)\Big)
\\
&
\qquad
\cdot_\enne \big(\big[u^{1}_{(2)} \cdots u^{k-q-j+1}_{(2)} [\psi(u^{k-q-j+2}_{(1)}, \ldots, u^{k-j+1}_{(1)})]_{(0)}\big]_{(0)}
 \cdot_\enne n_{(0)}\big), 
\\
&
\qquad\qquad\qquad
u^{1}_{(1)+(1)}, \ldots, u^{k-q-p+2}_{(1)+(1)}\Big)
\\
&
=
\Big(m_{(0)}, \Big(u^{1}_{(1)+(2)} \cdots u^{k-q-p+2}_{(1)+(2)} \gvf\Big(u^{k-q-p+3}_{(1)+}, \ldots, u^{k-q-j+1}_{(1)+}, 
\\
&
\qquad
[\psi(u^{k-q-j+2}_{(1)}, \ldots, u^{k-j+1}_{(1)})]_{(-2)+} u^{k-q-j+2}_{(2)+} \cdots u^{k-j+1}_{(2)+}, u^{k-j+2}_+, \ldots, u^k_+,
\\
&
\qquad
u^k_- \cdots u^{k-j+2}_- u^{k-j+1}_{(2)-} \cdots u^{k-q-j+2}_{(2)-}[\psi(u^{k-q-j+2}_{(1)}, \ldots, u^{k-j+1}_{(1)})]_{(-2)-} u^{k-q-j+1}_{(1)-} \cdots u^1_{(1)-}
\\
&
\qquad
u^{1}_{(2)+(1)} \cdots u^{k-q-j+1}_{(2)+(1)} 
[\psi(u^{k-q-j+2}_{(1)}, \ldots, u^{k-j+1}_{(1)})]_{(-1)}
u^{k-q-j+1}_{(2)-} \cdots u^{1}_{(2)-} 
n_{(-1)}
m_{(-1)}\Big)\Big)
\\
&
\quad
\cdot_\enne \big(u^{1}_{(2)+(2)} \cdots u^{k-q-j+1}_{(2)+(2)} [\psi(u^{k-q-j+2}_{(1)}, \ldots, u^{k-j+1}_{(1)})]_{(0)}
 \cdot_\enne n_{(0)}\big),
u^{1}_{(1)+(1)}, \ldots, u^{k-q-p+2}_{(1)+(1)}\Big)
\\
&
=
\Big(m_{(0)}, \Big(u^{1}_{+(2)} \cdots u^{k-q-p+2}_{+(2)} \gvf\Big(u^{k-q-p+3}_{+(1)}, \ldots, u^{k-q-j+1}_{+(1)}, 
\\
&
\qquad
[\psi(u^{k-q-j+2}_{(1)}, \ldots, u^{k-j+1}_{(1)})]_{(-1)} u^{k-q-j+2}_{(2)+} \cdots u^{k-j+1}_{(2)+}, u^{k-j+2}_+, \ldots, u^k_+,
\\
&
\qquad
u^k_- \cdots u^{k-j+2}_- u^{k-j+1}_{(2)-} \cdots u^{k-q-j+2}_{(2)-} u^{k-q-j+1}_{-} \cdots u^{1}_{-} 
n_{(-1)}
m_{(-1)}\Big)\Big)
\\
&
\qquad
\cdot_\enne \big(u^1_{+(3)} \cdots 
u^{k-q-p+2}_{+(3)}  u^{k-q-p+3}_{+(2)}  \cdots
u^{k-q-j+1}_{+(2)} 
\\
&
\qquad \qquad
[\psi(u^{k-q-j+2}_{(1)}, \ldots, u^{k-j+1}_{(1)})]_{(0)}
 \cdot_\enne n_{(0)}\big),
u^{1}_{+(1)}, \ldots, u^{k-q-p+1}_{+(1)}\Big)
\\
&
=
\Big(m_{(0)}, \Big(u^{1}_{+(2)} \cdots u^{k-q-p+2}_{+(2)} \Big(\gvf\big(u^{k-q-p+3}_{+(1)}, \ldots, u^{k-q-j+1}_{+(1)}, 
\\
&
\qquad
[\psi(u^{k-q-j+2}_{+(1)}, \ldots, u^{k-j+1}_{+(1)})]_{(-1)} u^{k-q-j+2}_{+(2)} \cdots u^{k-j+1}_{+(2)}, u^{k-j+2}_+, \ldots, u^k_+,
u^k_- \cdots 
u^{1}_{-} 
n_{(-1)}
m_{(-1)}\big)
\\
&
\qquad
\cdot_\enne \big(u^{k-q-p+3}_{+(2)} \cdots u^{k-q-j+1}_{+(2)}  [\psi(u^{k-q-j+2}_{+(1)}, \ldots, u^{k-j+1}_{+(1)})]_{(0)} \big) \Big)\Big)
 \cdot_\enne n_{(0)},
\\
&
\qquad \qqquad
u^{1}_{+(1)}, \ldots, u^{k-q-p+2}_{+(1)}\Big)
\\
&
=
\Big(m_{(0)}, \Big(u^{1}_{+(2)} \cdots u^{k-q-p+2}_{+(2)} \big((\gvf \circ_{j+1} \psi)\big(u^{k-q-p+3}_{+}, 
\ldots, u^k_+,
u^k_- \cdots 
u^{1}_{-} 
n_{(-1)}
m_{(-1)}\big)\big)\Big)
\\
&
\qquad
\cdot_\enne n_{(0)}, u^{1}_{+(1)}, \ldots, u^{k-q-p+2}_{+(1)}\Big)
\\
&
= (\gvf \circ_{j+1} \psi) \bullet_0 (m, n, u^1, \ldots, u^k).
\end{split}
\end{equation*}
\end{footnotesize}
\nopagebreak
In the first identity we just wrote down the definitions, 
and in all following ones we used the table of properties that hold for the translation map \rmref{tuscania2} and its compatibility with 
the coproduct \rmref{tuscania1}, see \cite[Eqs.~(2.4)--(2.12)]{KowKra:BVSOEAT}, or \cite[Prop.~3.7]{Schau:DADOQGHA} for the original reference. 
In the third identity we moreover used the fact that the Yetter-Drinfel'd algebra $N$ is in particular a comodule algebra, in the fourth identity we used the Yetter-Drinfel'd condition \rmref{huhomezone1b} for $N$, and in the sixth that $N$ is by definition also a module algebra. The seventh identity just uses the definition of the comp maps in \rmref{maxdudler1}, and the last one the definition \rmref{subiaco1} again.

As for proving \rmref{lagrandebellezza1} with respect to the operators \rmref{subiaco1}--\rmref{subiaco4}, the same comments as above hold: 
we only prove the case $i=0$ and $p \leq 1$, the cases for $i \geq 1$ (or $p=0$) being similar but easier. One has:
\begin{footnotesize}
\begin{equation*}
\begin{split}
& t\big(\gvf  \bullet_0 (m, n, u^1, \ldots, u^k)) \\
& =  t\big(m_{(0)}, \big(u^{1}_{+(2)} \cdots u^{k-p+1}_{+(2)} \gvf(u^{k-p+2}_+, \ldots, u^{k}_+, u^k_- \cdots u^1_- n_{(-1)} m_{(-1)})\big)
 \cdot_\enne n_{(0)}, 
 u^1_{+(1)}, \ldots, u^{k-p+1}_{+(1)}\big)
\\
&
=
\big(m_{(0)} u^1_{+(1)++}, 
u^1_{+(1)+-}\big[\big(u^{1}_{+(2)} \cdots u^{k-p+1}_{+(2)} \gvf(u^{k-p+2}_+, \ldots, u^{k}_+, u^k_- \cdots u^1_- n_{(-1)} m_{(-1)})\big)
 \cdot_\enne n_{(0)} \big]_{(0)}, 
\\
&
\qquad
 u^2_{+(1)+}, \ldots, u^{k-p+1}_{+(1)+}, 
\\
&
u^{k-p+1}_{+(1)-} \cdots u^1_{+(1)-} 
\big[\big(u^{1}_{+(2)} \cdots u^{k-p+1}_{+(2)} \gvf(u^{k-p+2}_+, \ldots, u^{k}_+, u^k_- \cdots u^1_- n_{(-1)} m_{(-1)})\big)
 \cdot_\enne n_{(0)} \big]_{(-1)} 
m_{(-1)}
\big)
\\
&
=
\big(m_{(0)} u^1_{+(1)++}, 
u^1_{+(1)+-}\big([u^{1}_{+(2)} \cdots u^{k-p+1}_{+(2)} \gvf(u^{k-p+2}_+, \ldots, u^{k}_+, u^k_- \cdots u^1_- n_{(-2)} m_{(-2)})]_{(0)}
 \cdot_\enne n_{(0)} \big), 
\\
&
\qquad
 u^2_{+(1)+}, \ldots, u^{k-p+1}_{+(1)+}, 
 \\
 &
 \qquad
u^{k-p+1}_{+(1)-} \cdots u^1_{+(1)-} 
[u^{1}_{+(2)} \cdots u^{k-p+1}_{+(2)} \gvf(u^{k-p+2}_+, \ldots, u^{k}_+, u^k_- \cdots u^1_- n_{(-2)} m_{(-2)})]_{(-1)} 
n_{(-1)} m_{(-1)}
\big)
\\
&
=
\big(m_{(0)} u^1_{+(1)++}, 
u^1_{+(1)+-}
\\
&
\qquad
\big(u^{1}_{+(2)+(2)} \cdots u^{k-p+1}_{+(2)+(2)} [\gvf(u^{k-p+2}_+, \ldots, u^{k}_+, u^k_- \cdots u^1_- n_{(-2)} m_{(-2)})]_{(0)}
 \cdot_\enne n_{(0)} \big), 
\\
&
\qquad
 u^2_{+(1)+}, \ldots, u^{k-p+1}_{+(1)+}, u^{k-p+1}_{+(1)-} \cdots u^1_{+(1)-} 
u^{1}_{+(2)+(1)} \cdots u^{k-p+1}_{+(2)+(1)} 
\\
&
\qquad
[\gvf(u^{k-p+2}_+, \ldots, u^{k}_+, u^k_- \cdots u^1_- n_{(-2)} m_{(-2)})]_{(-1)} 
\\
&
\qquad
u^{k-p+1}_{+(2)-} \cdots u^{1}_{+(2)-}
n_{(-1)} m_{(-1)}
\big)
\\
&
=
\big(m_{(0)} u^1_{+++}, 
u^{2}_{++(2)} \cdots u^{k-p+1}_{++(2)} [\gvf(u^{k-p+2}_+, \ldots, u^{k}_+, u^k_- \cdots u^1_- n_{(-2)} m_{(-2)})]_{(0)}
 \cdot_\enne (u^1_{++-} n_{(0)}), 
\\
&
\qquad
 u^2_{++(1)}, \ldots, u^{k-p+1}_{++(1)},
[\gvf(u^{k-p+2}_+, \ldots, u^{k}_+, u^k_- \cdots u^1_- n_{(-2)} m_{(-2)})]_{(-1)} 
\\
&
\qquad
u^{k-p+1}_{+-} \cdots u^{1}_{+-} n_{(-1)} m_{(-1)} 
\big)
\\
&
=
\big(m_{(0)} u^1_{++},
u^{2}_{+(2)} \cdots u^{k-p+1}_{+(2)} [\gvf(u^{k-p+2}_+, \ldots, u^{k}_+, u^k_- \cdots u^{k-p+2}_- u^{k-p+1}_{-(1)} \cdots u^1_{-(1)} n_{(-2)} m_{(-2)})]_{(0)}
\\
& 
\qqquad
 \cdot_\enne (u^1_{+-} n_{(0)}), 
\\
&
\qquad
 u^2_{+(1)}, \ldots, u^{k-p+1}_{+(1)},
[\gvf(u^{k-p+2}_+, \ldots, u^{k}_+, u^k_- \cdots u^{k-p+2}_- u^{k-p+1}_{-(1)} \cdots u^1_{-(1)} n_{(-2)} m_{(-2)})]_{(-1)} 
\\
&
\qquad
u^{k-p+1}_{-(2)} \cdots u^{1}_{-(2)} n_{(-1)} m_{(-1)} 
\big)
\\
&
= \gvf \bullet_1 (m_{(0)} u^1_{++}, u^1_{+-}n_{(0)}, u^2_+, \ldots, u^k_+, u^k_- \cdots u^1_- n_{(-1)} m_{(-1)})
\\
&
= \gvf \bullet_1 t(m, n, u^1, \ldots, u^k).
\end{split}
\end{equation*}
\end{footnotesize}
As before, we used the properties of the translation map (\cite[Eqs.~(2.4)--(2.12)]{KowKra:BVSOEAT}, \cite[Prop.~3.7]{Schau:DADOQGHA}) throughout, the comodule algebra condition for $N$ in the third identity, the Yetter-Drinfel'd condition \rmref{huhomezone1b} in the fourth equality, and the module algebra condition for $N$ in the fifth equality.

Finally, observe that the property \rmref{lagrandebellezza2} for $t$ from \rmref{subiaco4} already follows from Lemma \ref{altemps} in combination with \cite[Thm.~4.1]{KowKra:CSIACT}.
\end{proof}

With the prescriptions from \rmref{subiaco1}--\rmref{subiaco4}, one could now write down the operators $\iota_\gvf$, $\lie_\gvf$, and $S_\gvf$, but except for the case $N:=A$, where these operators were written down in \cite{KowKra:BVSOEAT}, the resulting expressions become considerably complicated and not enormously instructive, in striking contrast to the compact general expressions \rmref{alles4}, \rmref{messagedenoelauxenfantsdefrance2}, and \rmref{capillareal1}. However, using Theorems \ref{feinefuellhaltertinte}, \ref{waterbasedvarnish} \& \ref{calleelvira}, one now has at once:

\begin{corollary}
\label{dieda}
Let $(U,A)$ be a left Hopf algebroid, $M$ an anti Yetter-Drinfel'd module and $N$ a braided commutative Yetter-Drinfel'd algebra, and also let $M \otimes_\Aopp N$ be stable as anti Yetter-Drinfel'd module. 
Then the pair $\big(H^\bull(U,N), H_\bull(U,M \otimes_\Aopp N)\big)$ carries a canonical structure of a noncommutative differential calculus. In particular, if $U_\ract \in \moda$ is projective, then 
$\big(\Ext^\bull_\uhhu(A,N), \Tor^\uhhu_\bull(M,N)\big)$ can be canonically given the structure of a noncommutative differential calculus.
\end{corollary}

\begin{rem}
One might be tempted to think that the anti Yetter-Drinfel'd module $M \otimes_\Aopp N$ is stable whenever $M$ is; this, however, can be easily proven wrong by considering the case where $U$ is a Hopf {\em algebra} with involutive antipode over a commutative ground ring $k$, where $k$ is considered to be a stable anti Yetter-Drinfel'd module with trivial left coaction and right action given by the counit of the Hopf algebra. On the other hand, observe that to prove the properties \rmref{SchlesischeStr}--\rmref{lagrandebellezza1} 
the fact that $M$ is anti Yetter-Drinfel'd was {\em not} needed.
\end{rem}

\begin{rem}
In case the braided commutative Yetter-Drinfel'd algebra $N$ is given by the base algebra $A$ of $U$ (which is possible for any bialgebroid), the statement of Theorem \ref{zanjrevolution}, that is, the fact that $C_\bull(U,M)$ 
defines a para-cyclic module over the operad $C^\bull(U,A)$ was somewhat scattered through \cite[\S4]{KowKra:BVSOEAT}, especially Lemma 4.7, Lemma 4.11, and Lemma 4.15. This can be seen when using the convention that the extra comp module map $\gvf \bullet_0 -$ in this article corresponds to $\iota_\gvf s_{-1}$ in {\em op.~cit.}~and the standard comp module maps $\bullet_i$ correspond to $\bullet_{n-p-i+2}$, for $i = 1, \ldots, n-p+1$.
\end{rem}

\begin{rem}
Furthermore, in \cite[Rem.~4.8]{KowKra:BVSOEAT} the question arose whether it is possible to express the cap product $\iota_\gvf$ in Lemma 4.4 of {\em op.~cit.}~for bialgebroids by means of the operad multiplication $\mu$ and the comp module maps $\bullet_i$. Equation \rmref{alles4} together with \rmref{maxdudler3}, \rmref{subiaco3}, and \rmref{subiaco4} answers to this problem.
\end{rem}

\begin{rem}
Also, in the spirit of \cite{KowKra:BVSOEAT}, one could renounce on the cyclicity condition \rmref{lagrandebellezza2} and consider, for example, the more general case where $M$ is only a right $U$-module and a left $U$-comodule with coinciding respective 
$\Ae$-module structure, but with no anti Yetter-Drinfel'd condition. In this case, one has to replace the comp module $C_\bull(U,M)$ by its universal cyclic quotient $C_\bull(U,M)/(1 - t^{\bull +1})$ and the operad $C^\bull(U,N)$ by a certain subspace in which $\gvf \bullet_i (\mathrm{im}(1-t^{\bull+1})) \subseteq \mathrm{\im}(\id - t^{\bull+1})$, for all $i=1, \ldots, n-p+1$, to obtain a Batalin-Vilkovisky structure; see again \cite{KowKra:BVSOEAT}, in particular \S2.4 \& Def.~4.12, where this situation is discussed in case $N:=A$. 
\end{rem}

\begin{rem}
\label{genausowillich}
In \cite[\S3 \& \S4]{Kow:BVASOCAPB} it was explained how this example comprises the preceding two of associative algebras and (noncommutative) Poisson algebras by considering the left Hopf algebroid $(\Ae, A)$; in \cite[\S6]{KowKra:BVSOEAT} it was discussed how the general example in this section, by considering the universal enveloping algebra of a Lie-Rinehart algebra and dually its jet space, leads to the classical Cartan calculus of differential forms and vector fields by means of the classical de Rham differential, contraction, and Lie derivative.
\end{rem}

\providecommand{\bysame}{\leavevmode\hbox to3em{\hrulefill}\thinspace}
\providecommand{\MR}{\relax\ifhmode\unskip\space\fi M`R }
\providecommand{\MRhref}[2]{%
  \href{http://www.ams.org/mathscinet-getitem?mr=#1}{#2}
}
\providecommand{\href}[2]{#2}

\end{document}